\documentclass[12pt]{article}

\RequirePackage{geometry}
\usepackage{graphicx}				

\usepackage{amssymb}
\usepackage{amsmath}
\usepackage{color}
\usepackage{amsthm}
\usepackage{lineno}
\usepackage{indentfirst}
\usepackage{mathrsfs}
\usepackage{amsfonts}
\usepackage{amssymb}
\usepackage{geometry}
\usepackage{hyperref}
\usepackage{enumerate}
\usepackage{pgfplots}
\usepackage[sort&compress]{natbib}
\usetikzlibrary{patterns,arrows,positioning,calc,fadings,shapes,decorations.markings}
\usepackage{tikz}

\newtheorem{theorem}{\noindent Theorem}[section]
\newtheorem{lemma}[theorem]{\noindent Lemma}
\newtheorem{corollary}[theorem]{\noindent Corollary}
\newtheorem{proposition}[theorem]{\noindent Proposition}
\newtheorem{definition}[theorem]{\noindent Definition}
\newtheorem{assumption}{Assumption}[section]
\theoremstyle{definition}{
	\newtheorem{remark}[theorem]{\noindent Remark}
}
\numberwithin{equation}{section}
\newcommand{\R}{\mathbb{R}}

\newcommand{\E}{\mathscr{E}}
\newcommand{\G}{\mathscr{H}}
\newcommand{\Y}{\mathscr{Y}}
\newcommand{\W}{\mathcal{W}}

\newcommand{\spt}{\operatorname{spt}}

\hypersetup{hidelinks}
\begin{document}
	\title{Finite time blow-up and global solutions for the viscoelastic wave equation with combined power-type nonlinearities}
	\author{Haiyang Lin\footnote{Email address: linhaiyang2002@163.com}, Jinqi Yan\footnote{Email address: yjq20011027@163.com}, 
		Bo You\footnote{Corresponding author; Email address: youb2013@xjtu.edu.cn}\\
		{\small School of Mathematics and Statistics, Xi'an Jiaotong University, Xi'an, 710049, P. R. China}\\
Marcelo M. Cavalcanti\footnote{Email address: mmcavalcanti@uem.br}\\
	{\small Department of Mathematics, State University of
Maring\'a, Maring\'a, 87020-900,  P.  R, Brazil}		
		}
	\maketitle
\begin{center}
\begin{abstract}
The main objective of this paper is to investigate the global existence and nonexistence of solutions to the viscoelastic wave equation with a linear memory term of Boltzmann type, a nonlinear friction damping and a supercritical source term which is a combined power-type nonlinearities. The global existence of solutions is obtained provided that the energy sink dominates the energy source in an appropriate sense. In more general scenarios, we prove the global existence of solutions if the initial data is taken from a subset of a suitable potential well. Based on the global existence results, the energy decay rate is described in terms of the relaxation kernel as well as the growth order of the damping term. In addition, we also establish the blow-up results in the case that the source is stronger than the dissipative effect. In particular, we prove the finite time blow-up solutions exist at arbitrarily high initial energy.
			
\textbf{Keywords}: Potential well; Energy decay; Blow-up; Supercritical sources; Nonlinear wave equation; Viscoelasticity.
			
\textbf{Mathematics Subject Classification (2020)}: 35B40; 35B44; 35L15; 35L71.
\end{abstract}
\end{center}
\section{Introduction}
In this paper, we will consider the following viscoelastic wave equation 
\begin{equation}\label{equation}
\begin{cases}
u_{tt}-k(0)\Delta u-\int_{0}^{+\infty}k^{\prime}(s)\Delta u(t-s)\,ds+|u_t|^{m-1}u_t=f(u),\quad \text{in } \Omega\times (0,T),\\
u(x,t)=0,\quad \text{on }\Gamma\times (-\infty,T),\\
u(x,t)=u_0(x,t),\ u_t(x,t)=\partial_tu_0(x,t)\quad \text{in }\Omega\times (-\infty,0],
\end{cases}
\end{equation}
where $u(x,t)$ is a real valued unknown function on $\Omega\times (0,T),$ $\Omega\subset \R^3$ is a bounded domain with boundary $\Gamma$ of class $C^2.$ Notice that the results established in this paper remain valid for the bounded domains in $\mathbb{R}^n$ with $n\geq 3.$ The linear memory term $\int_{0}^{+\infty}k^{\prime}(s) \Delta u(t-s)\,ds$ of Boltzmann type quantifies the viscous resistance and provides a weak dissipation, the relaxation kernel satisfies that $k^{\prime}(s)<0$ for all $s>0$ and $k(+\infty)=1$. The term $|u_t|^{m-1}u_t$ with $m\ge 1$ represents a nonlinear damping which strongly dissipates energy and drives the system to stability. The nonlinear source term $f(u)$ is of the following form:
	\begin{align}\label{f(u)}
		f(u)=A(u)+B(u):=\sum_{i=1}^{r}a_i|u|^{p_i-1}u-\sum_{j=1}^{l}b_j|u|^{q_j-1}u
	\end{align}
with $a_i,$ $b_j>0,$ $1\le p_i,$ $q_j\le 5$ for $r,$ $l\geq 1,$ $A(u)$ is an energy source term that amplifies energy, $B(u)$ is an energy sink term that dissipates energy.

The global behavior of solutions to wave equations with nonlinear damping and source terms has been extensively studied. For example, the authors  \cite{Georgiev1994existence} have considered a three dimensional nonlinear wave equation with damping and sources
\begin{align}\label{eq:withoutmemory}
	u_{tt}-\Delta u+a|u_t|^{m-1}u_t=b|u|^{p-1}u\quad (a,b>0)
\end{align}
under a Dirichlet boundary condition with $1< p\le 3.$ They obtained the global existence of solution to problem \eqref{eq:withoutmemory} in the case that $1 < p \le m,$ and a blow-up result provided that the initial energy is sufficiently negative and $1<m<p$.  In \cite{Esquivel-Avila2003acharacterization}, the author extended some results to the nonlinear Kirchhoff equation 
\begin{align*}
	u_{tt}-M(\|\nabla u\|_2^2)\Delta u+a|u_t|^{m-1}u_t=b|u|^{p-1}u
\end{align*}
with homogeneous Dirichlet boundary condition. In a series of papers \cite{Barbu2005onnonlinear, Barbu2007blowup, Barbu2005existence}, the authors investigated the wave equations with degenerate damping and source terms
\begin{align*}
	u_{tt}-\Delta u+|u|^k|u_t|^{m-1}u_t=|u|^{p-1}u
\end{align*}
under a Dirichlet boundary condition, where the degenerate damping $|u|^k|u_t|^{m-1}u_t$ models friction modulated by strain. Moreover, there are some works (see  \cite{Vitillaro2002global, Cavalcanti2007wellposedness, Kou2024uniform}) about the wave equations with nonlinear boundary damping and source terms. Furthermore, the authors \cite{Cavalcanti2023decay} have also considered the stabilization and well-posedness for the semilinear wave equations with localized nonlinear dissipation. In addition, there are some results (see  \cite{Keith2006systems, Alves2009onexistence, Rammaha2010critically, Rammaha2010global, Guo2014systems}) about systems of wave equations.

In many practical applications, there are many materials exhibiting both viscous and elastic characteristics when they undergo deformation. Moreover, they also display time-dependent, rate-sensitive and history-dependent mechanical responses. To accurately describe wave propagation in media with hereditary characteristics, such as skin tissue, it is imperative to account for viscoelastic effects. As a precedent, Boltzmann \cite{boltzmann1909wissenschaftliche} characterized the behavior of viscoelastic materials through constitutive relations that involve long but fading memory. Subsequent foundational developments were established in \cite{Bernard1961viscoelasticity}, while comprehensive mathematical surveys can be found in \cite{Fabrizio1987mathematical, Renardy1987mathematical}. As we know, there are some results about the stability and asymptotic behavior of wave equations with hereditary memory. For example, the stability and asymptotic behavior of wave equations with hereditary memory was further analyzed in \cite{Dafermos1970anabstract, Dafermos1970asymptotic}. Subsequent advances in this line of research include a unified approach for deriving decay rates in integro-differential equations \cite{Alabau-Boussouira2008decay}, a general framework for optimal decay under broad kernel assumptions \cite{Alabau-Boussouira2009general}, and studies of abstract viscoelastic wave equations \cite{lasiecka2013note}. Moreover, the introduction of frictional damping and a nonlinear source term into the evolution equations leads to some new challenges in the mathematical investigation of such models. Based on the framework of viscoelasticity, there are some works  about the global existence and nonexistence for the wave equations with source term in the case that the growth order $p$ of the source term satisfies the subcritical $1\leq p<3$ and critical $p=3$ cases (see \cite{Cavalcanti2003frictional, Cavalcanti2016attractors, Song2015blow, Messaoudi2003blow, Messaoudi2006blow}), the more challenging supercritical case $3< p\le 6$ (see \cite{Guo2014hadamard,Guo2017blowup,Guo2018energy}). In the existing literatures, it is worthy mentioning that the Hadamard well-posedness for problem \eqref{equation} was obtained in \cite{Guo2014hadamard}, while  the global behavior of solutions for problem \eqref{equation} with $f(u)=|u|^{p-1}u$ was analyzed in \cite{Guo2017blowup,Guo2018energy}. Based on these works in \cite{Guo2014hadamard,Guo2017blowup}, the authors \cite{Sun2019blow} have also proved the blow-up results with arbitrary initial energy. The global behavior for wave equations with localized memory, frictional damping and a $p$-homogeneous source term $f(u)$ was considered in \cite{Cavalcanti2023aymptotic}. However, the existing literatures mentioned above are mainly focused on the models with homogeneous source terms, there is relatively few results (see \cite{Tao2011thenonlinear, Kutev2014global, Dinh2019onfractional, Miao2013thedynamics, Dimova2024global}) about the evolution equations with the combined power-type nonlinearities. To the authors' knowledge, there is no results about the global existence and nonexistence of problem \eqref{equation} with the combined power-type nonlinearities, which motivates the present work.

The main objective of this paper is to study the finite time blow-up and global existence of solutions for problem \eqref{equation} with combined nonlinearities. In what follows, let us comment on the main difficulties and novelties of this paper:
\begin{enumerate}[\rm{(}1\rm{)}]
\item The presence of the memory term motivates us to define the potential well in the weighted Hilbert space $L^2_{\mu}(\mathbb{R}^+;H^1_0(\Omega))$, which is specifically suited for the full past history problem \eqref{equation}, rather than in the standard space $H^1_0(\Omega)$. Additionally, the blow-up analysis for problem \eqref{equation} becomes more subtle, particularly in imposing a constraint on the upper bound of the initial energy.
\item In our model, the source term is non-homogeneous, which is distinct from other relevant models. Thus, the following estimate
\begin{align*}
c\|u\|_{p+1}^{p+1}\le\int_{\Omega}f(u)u\,dx\le C\|u\|_{p+1}^{p+1}
\end{align*}
for some positive constants $c,$ $C$ can not be obtained for some $1\le p\le5$, which would pose some technical difficulties in the proof of blow-up analysis. Furthermore, the construction of the potential well is complicated by the energy sink $B(u)$, whose dissipative effect can only be partially utilized. Moreover, the depth $d$ of the potential well defined in \eqref{d0} is different from the reference value $d_0$ defined in \eqref{d00} when $r>1$. As a consequence we have to impose some tighter constraints on the initial energy to establish the energy decay result.
\item The frictional damping, memory term and energy sink dissipate energy, while the energy source amplifies energy. It would be quite interesting to determine whether the system is stable and explore the mechanism of how these terms interplay to lead to blow-up or decay of the energy. The main features of such model consist of a supercritical source term, a memory term accounting for the full past history and a frictional damping term with an arbitrary growth rate $m$. The only exception to the arbitrariness of $m$ is reserved for the blow-up analysis, where we impose the constraint $m\le p_r\le 5$.
\item A comprehensive analysis of the blow-up of solutions to problem \eqref{equation} is presented in this paper. We establish the blow-up of solutions with negative initial energy and nonnegative energy under certain constraints. Furthermore, we prove that for arbitrary initial energy there exists a solution that blows up in finite time.
\end{enumerate}
The rest of this paper is organized as follows. In Section \ref{Pre}, we present the main assumptions throughout the paper and review some well-posedness results. In Section \ref{potentialwellsec}, we introduce the potential well and recall some useful lemmas. In Section \ref{glb}, we first prove the global existence when the energy source dominates the energy sink, and then establish the global existence results in the case that the initial data lies in the potential well. In Section \ref{decay}, we establish the energy decay results under the different conditions of relaxation kernel. In Section \ref{blo}, we analyze the finite time blow-up of solutions for problem \eqref{equation} when the source term is stronger than the dissipative effect.
	
Throughout this paper, we denote the standard $L^p(\Omega)$-norm by $\|\cdot\|_p=\|\cdot\|_{L^p(\Omega)},$ $\R^+=[0,+\infty).$ Let $C$ be a generic positive constants which may change from one line to another line, even in the same line. 
\section{Preliminary}\label{Pre}
In this section, we will introduce the main assumptions in this paper and recall some well-posedness results established in \cite{Guo2014hadamard}.

To define a proper function space for initial data, we denote by $\mu(s)=-k^\prime(s)$ and let the weighted Hilbert space $L^2_{\mu}(\R^+;H^1_0(\Omega))$ be the collection of all functions $u:[0,+\infty)\to H^1_0(\Omega)$ satisfying $$\int_{0}^{+\infty} \int_{\Omega} \mu(s) |\nabla u(x,s)|^2 dx ds<+\infty.$$ In particular, the space $L^2_{\mu}(\R^{-};H_0^1(\Omega))$ consists of all functions $u:(-\infty,0]\to H^1_0(\Omega),$ such that $u(-t)\in L^2_{\mu}(\R^+;H_0^1(\Omega))$. Similarly, we can define the weighted Hilbert spaces $L^2_{\mu}(\R^{+};L^2(\Omega))$ and $L^2_{\mu}(\R^{-};L^2(\Omega))$.
	 
The following assumptions will be imposed throughout this paper.
\begin{assumption}\label{Assumption1}
\quad
\begin{itemize}
\item The function $f(u)$ is defined as in \eqref{f(u)}, where  $1< p_1<\cdots<p_r\le 5$ and $1\le q_1<\cdots<q_l\le 5$;
\item $m\ge 1,$ $\frac{m+1}{m}p_r<6$ and $\frac{m+1}{m}q_l<6$;
\item The relaxation kernel $k\in C^2(\R^{+})$ satisfying $k^{\prime}(s)<0$ for all $s>0$ and $k(\infty)=1$;
\item The function $\mu(s)=-k^{\prime}(s)\in C^1(\R^+)\cap L^1(\R^+)$ satisfying $\mu^{\prime}(s)\le 0$ for all $s>0$ and $\mu(+\infty)=0$;
\item The function $u_0:\R^-\to H^1_0(\Omega)$ satisfying $u_0\in L^2_{\mu}(\R^-;H^1_0(\Omega)),$ the function $\partial_tu_0:\R^-\to L^2(\Omega)$ and $\partial_tu_0\in L^2_{\mu}(\R^-;L^2(\Omega))$ are weakly continuous at $t=0$. 
\end{itemize}
\end{assumption}
For convenience, we denote by 
\begin{align*}
	p_0:=\max\{p_r,q_l\},
\end{align*} 
then we have
\begin{align*}
p_0\frac{m+1}{m}<6\quad \text{and}\quad |f^{\prime}(u)|\le C(1+|u|^{p_0-1}).
\end{align*}
We also make the following assumption that $p_i\neq q_j$ for any $i\in \{1,\cdots,r\}$ and $j\in \{1,\cdots,l\}$. Indeed, if $p_i=q_j$ for some $i$ and $j$, the terms $a_i|u|^{p_i-1}u$ and $-b_j|u|^{q_j-1}u$ can be written into the following form $(a_i-b_j)|u|^{p_i-1}u$.
	 
In what follows, we will give the definition of weak solution to problem \eqref{equation}.
\begin{definition}\label{weaksolution}
Let $T>0$ be fixed. A function $u(x,t)$ is said to be a weak solution of problem \eqref{equation} defined on the time interval $(-\infty,T],$ if $u\in C([0,T];H_0^1(\Omega))$ with $u_t\in C([0,T];L^2(\Omega))\cap L^{m + 1}(\Omega\times(0,T))$ and
\begin{itemize}
\item $u(x,t)=u_0(x,t)$ for $t\le 0$;
\item For any $t\in[0,T]$ and any test function $\phi\in\mathscr{F},$ we have
\begin{align}\label{solutioneq}
\nonumber&\int_{\Omega}u_t(x,t)\phi(x,t)\,dx-\int_{\Omega}u_t(x,0)\phi(x,0)\,dx-\int_0^t\int_{\Omega}u_t(x,\tau)\phi_t(x,\tau)\,dxd\tau\\
\nonumber&+k(0)\int_0^t\int_{\Omega}\nabla u(x,\tau)\cdot\nabla\phi(x,\tau)\,dxd\tau+\int_0^t\int_0^{+\infty}\int_{\Omega}\mu(s)\nabla u(x,\tau-s)\cdot\nabla\phi(x,\tau)\,dxdsd\tau\\
&+\int_0^t\int_{\Omega}|u_t(x,\tau)|^{m-1}u_t(x,\tau)\phi(x,\tau)\,dxd\tau=\int_0^t\int_{\Omega}f(u(x,\tau))\phi(x,\tau)\,dxd\tau,
\end{align}
where
\begin{align*}
\mathscr{F}=\left\{\phi\in C([0,T];H_0^1(\Omega))\cap L^{m + 1}(\Omega\times(0,T)):\phi_t\in C([0,T];L^2(\Omega))\right\}.
\end{align*}
\end{itemize}
\end{definition}
Let us introduce the history function
\begin{align*}
w^t(x,s)=u(x,t)-u(x,t-s),\quad s\ge 0,
\end{align*}
and define the quadratic energy functional 
\begin{align}\label{quadraticenergyfunctional}
\E(t):=\frac{1}{2}\left(\|u_t(t)\|_2^2+\|\nabla u(t)\|_2^2+\int_{0}^{+\infty}\mu(s)\|\nabla w^t(s)\|^2_2\,ds\right).
\end{align}
In what follows, we will recall the well-posedness of problem \eqref{equation}.
	\begin{theorem}[See \cite{Guo2014hadamard}]\label{thm:1}
Assume that Assumption \ref{Assumption1} holds. Then there exists a local (in time) weak solution $u$ to problem (1.1) defined on the time interval $(-\infty, T]$ for some $T>0$ depending on the initial quadratic energy $\E(0).$ Furthermore, the following energy identity holds on $[0,T]$:
\begin{align}\label{energyidentity}
\nonumber&\E(t)+\int_{0}^{t}\|u_t(\tau)\|_{m+1}^{m+1}\,d\tau-\frac{1}{2}\int_{0}^{t}\int_{0}^{+\infty}\mu'(s)\|\nabla w^{\tau}(s)\|_{2}^{2}\,ds d\tau\\
=&\E(0)+\int_{0}^{t}\int_{\Omega}f(u(x,\tau))u_t(x,\tau)\,dxd\tau.
\end{align}
In addition, if $u_0(0)\in L^{\frac{3(p_0-1)}{2}}(\Omega)$, then $u$ is a unique solution of problem \eqref{equation}.
\end{theorem}
\begin{remark}\label{rmk:2}
More precisely, the lifespan $T$ in Theorem \ref{thm:1}  depends only on an upper bound of $\E(0),$ i.e., $T=T(M)$ for any initial data satisfying $\E(0)\leq M$ for some $M >0.$
	\end{remark}
Since $u(t)\in H^1_0(\Omega)\subset L^{p_0+1}(\Omega)$ for all $t\in [0,T),$ where $T$ is the lifespan of the solution, we can define the total energy functional
\begin{align*}
E(t)=\E(t)-\int_{\Omega}F(u(t))\,dx,\,\,\forall\,\,t\in [0,T),
\end{align*}
where $F(u)=\int_{0}^{u}f(s)\,ds.$ Then the energy identity \eqref{energyidentity} can be rewritten into the following form
	\begin{align}\label{Eenergyidentity}
		E(t)+\int_{0}^{t}\int_{\Omega}|u_t(x,\tau)|^{m+1}\,dxd\tau-\frac{1}{2}\int_{0}^t\int_{0}^{+\infty}\mu^{\prime}(s)\|\nabla w^{\tau}(s)\|_2^2\,dsd\tau=E(0).
\end{align}
Moreover, we also have
\begin{align}\label{E'}
E^{\prime}(t)=-\|u_t(t)\|_{m+1}^{m+1}+\frac{1}{2}\int_{0}^{+\infty}\mu^{\prime}(s)\|\nabla w^t(s)\|_2^2\,ds\le 0,
\end{align}
which implies that the total energy $E(t)$ is non-increasing in time.
	
	We would also like to state the global existence results in the case that $m\ge p_0$.
	\begin{theorem}[See \cite{Guo2014hadamard}]
Under the same assumptions as in Theorem \ref{thm:1},		if $u_0(0)\in L^{p_0+1}(\Omega)$ and $m\ge p_0$, then the weak solution of problem \eqref{equation} globally exists.
	\end{theorem}
	\section{Potential well}\label{potentialwellsec}
	This section is devoted to a detailed discussion of the potential well and its depth, which will play a crucial role in this paper. Inspired by the idea in \cite{Guo2018energy}, we introduce a suitable notion of a potential well for problem \eqref{equation}. Let $l_0$ be the index such that $q_{l_0}<p_1<q_{l_0+1}$ if it exists. Otherwise, let $l_0=0,$ if $p_1<q_1$ and $l_0=l,$ if $q_l<p_1.$ Define the functional 
\begin{align*}
I(v)=&\frac{1}{2}\left(\|\nabla v(0)\|_2^2+\int_{0}^{+\infty}\mu(s)\|\nabla v(0)-\nabla v(-s)\|_2^2\,ds\right)\\
&-\sum_{i=1}^{r}\frac{a_i}{p_i+1}\|v(0)\|_{p_i+1}^{p_i+1}+\sum_{j=1}^{l_0}\frac{b_j}{q_j+1}\|v(0)\|_{q_j+1}^{q_j+1}
\end{align*}
for $v\in L^2_{\mu}(\R^-;H^1_0(\Omega))$ which is weakly continuous at $t=0$, and let the constant $d$ be the depth of potential well defined by
	\begin{align}\label{d0}
		d:=\inf_{v\in \mathcal{M}}I(v),
	\end{align}
	where $\mathcal{M}$ is the Nehari manifold given by
	\begin{align*}
		\mathcal{M}:=\Big\{&v\in L^2_{\mu}(\R^{-};H^1_0(\Omega))\backslash\{0\}:v\ \text{is weakly continuous at $t=0$, and}\\
		&\|\nabla v(0)\|_2^2+\int_{0}^{+\infty}\mu(s)\|\nabla v(0)-\nabla v(-s)\|_2^2\,ds+\sum_{j=1}^{l_0}b_j\|v(0)\|_{q_j+1}^{q_j+1}=\sum_{i=1}^{r}a_i\|v(0)\|_{p_i+1}^{p_i+1}\Big\}.
	\end{align*}
In what follows, we introduce the potential well
	\begin{align*}
		\W:=\{v\in L^2_{\mu}(\R^{-};H^1_0(\Omega)):v\ \text{is weakly continuous at}\ t=0,\ I(v)<d\}
	\end{align*}
and for $p_i\in [1,5],$ denote by $\gamma_i$ the best constant for the Sobolev inequality given by
	\begin{align}\label{gamma}
		\gamma_i=\sup\{\|u\|_{p_i+1}:u\in H^1_0(\Omega),\|\nabla u\|_2=1\}<+\infty.
	\end{align}
Moreover, we also define a function $G:\R^+\to \R$ by
	\begin{align*}
		G(y)=y-\sum_{i=1}^{r}\frac{a_i}{p_i+1}(2\gamma_i^2y)^{\frac{p_i+1}{2}}.
	\end{align*} 
	Suppose that the function $G(y)$ attains its maximum at $y=y_0$ and denote the corresponding maximum value of $G(y)$ by $d_0$. Obviously, $y_0$ satisfies
\begin{align}\label{y0}
\sum_{i=1}^{r}a_i\left(2\gamma_i^2\right)^{\frac{p_i+1}{2}}y_0^{\frac{p_i-1}{2}}=2
\end{align}
and 
\begin{align}\label{d00}
d_0=G(y_0)=\sum_{i=1}^{r}\left(\frac{1}{2}-\frac{1}{p_i+1}\right)a_i\left(2\gamma_i^2y_0\right)^{\frac{p_i+1}{2}}>0.
\end{align}
In what follows, we present several lemmas to better characterize $d$ and $d_0$.
\begin{lemma}\label{lem:3.3}
The constant $d$ defined as in \eqref{d0} is nonnegative and coincides with the mountain pass level, that is
	\begin{align}\label{d0m}
		d=\inf_{v(0)\neq0}\sup_{\lambda> 0}I(\lambda v)\ge0. 
	\end{align}
\end{lemma}
	\begin{proof}
		
By carrying out some simple calculation, we obtain
\begin{align*}
\partial_{\lambda}I(\lambda v)=\lambda\|\nabla v(0)\|_2^2+\lambda\int_{0}^{+\infty}\mu(s)\|\nabla v(0)-\nabla v(-s)\|_2^2\,ds\\-\sum_{i=1}^{r}\lambda^{p_i}a_i\|v(0)\|_{p_i+1}^{p_i+1}+\sum_{j=1}^{l_0}\lambda^{q_j}b_j\|v(0)\|_{q_j+1}^{q_j+1}.
\end{align*}
Notice that for any fixed $v\in L^2_{\mu}(\R^-;H^1_0(\Omega))$ which is weakly continuous at $t=0$ with $v(0)\neq 0$, the equation $\partial_{\lambda}I(\lambda v)=0$ has a unique root $\lambda_0$ on $(0,\infty)$ since $q_{l_0}<p_1<\cdots <p_r$, i.e.,
\begin{align*}
&\lambda_0^2\|\nabla v(0)\|_2^2+\lambda_0^2\int_{0}^{+\infty}\mu(s)\|\nabla v(0)-\nabla v(-s)\|_2^2\,ds\\
=&\sum_{i=1}^{r}\lambda_0^{p_i+1}a_i\|v(0)\|_{p_i+1}^{p_i+1}-\sum_{j=1}^{l_0}\lambda^{q_j+1}_0b_j\|v(0)\|_{q_j+1}^{q_j+1},
\end{align*}
which implies that $\lambda_0v\in \mathcal{M}$. In addition, the maximum value of $I(\lambda v)$ for $\lambda>0$ is positive and is achieved at $\lambda_0>0$. Then, we have
\begin{align*}
d=\inf_{u\in \mathcal{M}}I(u)\le I(\lambda_0v)=\sup_{\lambda>0}I(\lambda v).
\end{align*}
Therefore, we obtain
\begin{align*}
d\le \inf_{v(0)\neq0}\sup_{\lambda>0}I(\lambda v).
\end{align*}
On the other hand, for any $v\in \mathcal{M}$, we claim that $v(0)\neq0$. Indeed, if $v(0)=0$, it follows from the definition of $\mathcal{M}$ that
\begin{align*}
\int_{0}^{+\infty}\mu(s)\|\nabla v(-s)\|_2^2\,ds=0.
\end{align*}
Thus, we obtain $v=0,$ which contradicts with $v\in \mathcal{M}$. Therefore, we deduce from the above argument that $$0<\sup_{\lambda>0}I(\lambda v)=I(v)$$ for any $v\in \mathcal{M}.$ Thus, we obtain
		\begin{align*}
			d=\inf_{u\in \mathcal{M}}I(u)=\inf_{u\in \mathcal{M}}\sup_{\lambda>0}I(\lambda u)\ge \inf_{v(0)\neq0}\sup_{\lambda>0}I(\lambda v)\ge 0.
		\end{align*}
Hence, we have $$d=\inf_{v\neq0}\sup_{\lambda>0}I(\lambda v).$$
	\end{proof}
	\begin{lemma}\label{lem:3.4}
		The constant $d$ also coincides with the mountain pass level of $J$, i.e.,
\begin{align*}
d=\inf_{u\in H^1_0(\Omega)\backslash\{0\}}\sup_{\lambda>0}J(\lambda u),
\end{align*}
where 
\begin{align*}
J(u)=\frac{1}{2}\|\nabla u\|_2^2-\sum_{i=1}^{r}\frac{a_i}{p_i+1}\|u\|_{p_i+1}^{p_i+1}+\sum_{j=1}^{l_0}\frac{b_j}{q_j+1}\|u\|_{q_j+1}^{q_j+1}.
\end{align*}
\end{lemma}
\begin{proof}
It follows from the definition of $I$ and $J$ that
\begin{align*}
\sup_{\lambda>0}I(\lambda v)\ge \sup_{\lambda>0}J(\lambda v(0)).
\end{align*}
Thus, we obtain
\begin{align}\label{eq:3.5}
\inf_{v(0)\neq0}\sup_{\lambda>0}I(\lambda v)\ge \inf_{v(0)\neq0}\sup_{\lambda>0}J(\lambda v(0))=\inf_{u\in H^1_0(\Omega)\backslash\{0\}}\sup_{\lambda>0}J(\lambda u).
\end{align}
On the other hand, by carrying out the similar proof of Lemma \ref{lem:3.3}, we obtain
\begin{align}\label{Jd}
\inf_{u\in H^1_0(\Omega)\backslash\{0\}}\sup_{\lambda>0}J(\lambda u)=\inf_{u\in \mathcal{N}}J(u),
\end{align}
where 
\begin{align*}
\mathcal{N}=\{u\in H^1_0(\Omega)\backslash\{0\}:\|\nabla u\|_2^2+\sum_{j=1}^{l_0}b_j\|u\|_{q_j+1}^{q_j+1}=\sum_{i=1}^{r}a_i\|u\|_{p_i+1}^{p_i+1}\}.
\end{align*}
Denote by $$v(t)=u,\quad t\in \R^-.$$ Since $u\in \mathcal{N}$ yields $v\in \mathcal{M}$, we obtain
\begin{align}\label{eq:3.7}
\inf_{u\in \mathcal{N}}J(u)=\inf_{v(t)=u,\ u\in \mathcal{N}}I(v)\ge \inf_{v\in \mathcal{M}}I(v)=d.
\end{align}
Combining \eqref{eq:3.5}-\eqref{eq:3.7}, we deduce 
\begin{align*}
	d=\inf_{u\in H^1_0(\Omega)\backslash\{0\}}\sup_{\lambda>0}J(\lambda u).
\end{align*}
\end{proof}
\begin{lemma}\label{lem:3.5}
The constants $d$ and $d_0$ are defined in \eqref{d0} and \eqref{d00} respectively. Then we have $$d\ge d_0>0.$$
\end{lemma}
\begin{proof}
We deduce from Sobolev inequality and the definition of $G$ that
\begin{align}\label{Iv>}
J(\lambda u)\ge \frac{1}{2}\|\lambda\nabla u\|_2^2-\sum_{i=1}^r\frac{a_i\gamma_i^{p_i+1}}{p_i+1}\|\lambda\nabla u\|_{2}^{p_i+1}= G\left(\frac{\|\lambda\nabla u\|^2_2}{2}\right).
\end{align}
Thus, we obtain
\begin{align*}
d=\inf_{u\neq0}\sup_{\lambda>0}J(\lambda u)\ge \inf_{u\neq0}\sup_{\lambda>0}G\left(\frac{\|\lambda\nabla u\|_2^2}{2}\right)=d_0>0.
\end{align*}
\end{proof}
\begin{remark}
From the proof of Lemma \ref{lem:3.5}, we notice that $d=d_0$ if and only if $l_0=0$ and there exists a sequence $\{u_n\},$ such that
\begin{align}\label{condit}
\gamma_i=\lim_{n\to \infty}\frac{\|u_n\|_{p_i+1}}{\|\nabla u_n\|_2},\quad i=1,\cdots,r.
\end{align}
However, in general, condition \eqref{condit} does not hold apart from the special case $r=1$. Consequently, when $l_0=0$ and $r=1$, we have $d=d_0$; otherwise, we only have $d\ge d_0$ and the equality generally does not hold.
\end{remark}
Next, we define two subsets $\mathcal{W}_1$ and $\mathcal{W}_2$ of the potential well, which will be used in subsequent sections:
\begin{align}
&\W_1=\left\{v\in \W:I_0(v)>0\right\}\bigcup\{0\},\label{W_1}\\
&\W_2=\left\{v\in \W:I_0(v)<0\right\},\label{W_2}
\end{align}
where 
\begin{align*}
I_0(v)=\|\nabla v(0)\|_2^2+\int_{0}^{+\infty}\mu(s)\|\nabla v(0)-\nabla v(-s)\|_2^2\,ds-\sum_{i=1}^ra_i\|v(0)\|_{p_i+1}^{p_i+1}+\sum_{j=1}^{l_0}b_j\|v(0)\|_{q_j+1}^{q_j+1}.
\end{align*}
Obviously, we have $\W_1\bigcap\W_2=\emptyset$. Moreover, due to $$I(v)<\inf_{u\in \mathcal{M}}I(u),\quad \forall v\in \W,$$ we have $\mathcal{M}\bigcap\W=\emptyset$, i.e.,
	\begin{align*}
		\{v\in \W:I_0(v)=0\}=\{0\}.
	\end{align*}
	Thus, we also have $$\W_1\bigcup\W_2=\W.$$
	
	\section{Global existence}\label{glb}
In this section, we will establish the global existence results. To begin with, we introduce a well-known lemma used in the sequel.
\begin{lemma}\label{lem:1}
Assume that Assumption \ref{Assumption1} holds, then the weak solution to problem \eqref{equation} is either global or there exists some time $0<T_{max}<+\infty,$ such that
\begin{align*}
\limsup_{t\to T_{max}^-}\E(t)=+\infty.
\end{align*}
\end{lemma}
\begin{proof}
If $u$ is a weak solution to problem \eqref{equation} on $[0,T_0]$ that is not global and there exists a constant $M>0,$ such that $\E(t)\le M$ for all $t\in [0,T_{max})$, where $T_{max}<+\infty$ is the maximum lifespan of $u$ defined by
\begin{align}\label{lifespan}
T_{max}:=\sup\left\{T:\text{there exists a weak solution } \tilde{u} \text{ of \eqref{equation} on $[0,T]$ such that $\tilde{u}=u$ on $[0,T_0]$}\right\}.
\end{align}
It follows from Theorem \ref{thm:1} and Remark \ref{rmk:2} that there exists a $\tilde{T}(M) > 0,$ such that the problem \eqref{equation} admits a unique solution on $(-\infty, \tilde{T}(M)]$ for $\E(0)\le M.$ Now, consider a sequence $t_j\to T_{max}^-$ and let $k$ be a integer such that $t_k+\tilde{T}(M)>T_{max}$. Then, we conclude from Theorem \ref{thm:1} and $\E(t_k) \le M$ that we can extend the solution $u$ of problem \eqref{equation} with $u_0(x,t)=u|_{(-\infty,t_k]}$ up to the time $t_k + \tilde{T}(M).$ This is a contradiction with the definition of $T_{max}.$
	\end{proof}
	Now, we will prove the global existence of the solution when $p_r<q_l$.
	\begin{theorem}\label{thm:p<q}
Suppose that Assumption \ref{Assumption1} is in force and $p_r<q_l$, then the weak solution of problem \eqref{equation} is global.
\end{theorem}
\begin{proof}
From H\"{o}lder's inequality and Young's inequality, we deduce that there exists a positive constant $C,$ such that
\begin{align*}
\int_{\Omega}F(u(t))\,dx=&\sum_{i=1}^{r}\frac{a_i}{p_i+1}\|u\|_{p_i+1}^{p_i+1}-\sum_{j=1}^{l}\frac{b_j}{q_j+1}\|u\|^{q_j+1}_{q_j+1}\\
\le& \frac{b_l}{q_l+1} \|u\|_{q_l+1}^{q_l+1}+C-\sum_{j=1}^{l}\frac{b_j}{q_j+1}\|u\|^{q_j+1}_{q_j+1}\\
\le& C.
\end{align*}
It follows from inequality \eqref{E'} that
\begin{align*}
\E(t)=E(t)+\int_{\Omega}F(u(t))\,dx\le E(0)+C.
\end{align*}
Hence, we infer from Lemma \ref{lem:1} that $T_{max}=+\infty$, i.e., $u$ is a global weak solution of problem \eqref{equation}. 
\end{proof}
	In the general case, it is very difficult to prove the weak solution of problem \eqref{equation} is global. To overcome the difficulty, we have to impose some restrictions on initial energy and make use of potential well method introduced in \cite{Sattinger1968onglobal,Payne1975saddle}. For more details about potential well method, we refer the readers to \cite{Yacheng2006onpotential,Esquivel-Avila2014blowup,Lian2019global,Dimova2024global,Cavalcanti2023aymptotic,Xu2018global,Ding2020global,Ding2020local} and the references therein.
	
	In Section \ref{potentialwellsec}, we have analyzed the structure of potential well and introduced the subsets $\mathcal{W}_1$ and $\mathcal{W}_2.$ Now, we will show that if $u_0\in \mathcal{W}_1$ and $E(0)<d$, then the weak solution of problem \eqref{equation} is global.
\begin{theorem}\label{globalex}
Assume that Assumption \ref{Assumption1} holds and let $u$ be the weak solution of problem \eqref{equation} on its maximal existence interval $(-\infty,T).$ If $E(0)<d$ and $u_0\in \W_1$. Then the solution $u$ is global, i.e., $T=+\infty$. Furthermore, we have
\begin{align}\label{Ecompar}
0\le E(t)\le E(0)<d\quad \text{and}\quad0\le \frac{p_1-1}{p_1+1}\E(t)\le E(t)
\end{align} 
for all $t\ge0.$
\end{theorem}
\begin{proof}
Define the shift of $u$ by
\begin{align*}
u^t(\tau):=u(t+\tau),\ \tau\in \R^{-},
\end{align*}
it follows from \eqref{E'} that $E(t)$ is non-increasing, which implies that 
\begin{align}\label{<d0}
I(u^t)\le E(t)\le E(0)< d,\quad t\in (0,T).
\end{align}
Hence, we conclude that $u^t\in \W$ for all $t\in [0,T)$.

We claim that $u^t\in \W_1$ for all $t\in (0,T)$. If $u_0=0$, then we infer from the uniqueness of the solution of problem \eqref{equation} that $u(t)\equiv 0$ for any $t\in (0,T),$ which entails that $u^t\in \W_1$ for all $t\in (0,T)$. Thus, we only need to consider the case that $u_0\neq 0.$ We argue by contradiction. Assume that there exists $t_1\in (0,T),$ such that $u^{t_1}\notin \W_1.$ Thus, we obtain $u^{t_1}\in \W_2.$ We conclude from the energy identity \eqref{energyidentity} that $\E(t)$ is continuous in $t,$ which along with $u\in C([0,T];H^1_0(\Omega))\bigcap C^1([0,T];L^2(\Omega))$ and 
\begin{align*}
\int_{0}^{+\infty}\mu(s)\|\nabla w^t(s)\|_2^2\,ds=2\E(t)-\|u_t(t)\|_2^2-\|\nabla u(t)\|_2^2,
\end{align*} 
we conclude that $\int_{0}^{+\infty}\mu(s)\|\nabla w^t(s)\|_2^2\,ds$ is continuous in $t.$ Moreover, we infer from $u\in C([0,T];H^1_0(\Omega))\subset C([0,T];L^{p_0+1}(\Omega))$ that $\|u(t)\|_{p_i+1}^{p_i+1}$ and $\|u(t)\|_{q_j+1}^{q_j+1}$ are continuous in $t$. Therefore, we conclude that $\mathcal{I}(t):=I_0(u^t)$ is continuous in $t$. Note that $\mathcal{I}(0)>0$ since $u_0\in \W_1$ and $\mathcal{I}(t_1)<0$ due to $u^{t_1}\in\W_2$. By applying intermediate value theorem, we deduce that there exists some time $t_0\in (0,t_1),$ such that $\mathcal{I}(t_0)=0$. Thus, we obtain $u^{t_0}\in \mathcal{M}$ or $u^{t_0}=0$. If $u^{t_0}=0$, we have $u_0(\tau)=u^{t_0}(\tau-t_0)=0$ for a.e. $\tau\le 0$, which is contradict with $u_0\neq 0$. If $u^{t_0}\in \mathcal{M}$, we obtain
\begin{align*}
I(u^{t_0})\ge d>E(0)\ge E(t_0)\ge I(u^{t_0}).
\end{align*}
Hence, we arrive at a contradiction. Thus, we conclude that $u^t\in \W_1$ for all $t\in (0,T)$.
		
Now, we deduce from $u^t\in \W_1$ that for $t\in [0,T],$
\begin{align*}
\sum_{i=1}^{r}a_i\|u(t)\|_{p_i+1}^{p_i+1}&<\|\nabla u(t)\|_2^2+\int_{0}^{+\infty}\mu(s)\|\nabla w^t(s)\|_2^2\,ds+\sum_{j=1}^{l_0}b_j\|u(t)\|_{q_j+1}^{q_j+1}\\&\le 2\E(t)+\sum_{j=1}^{l_0}b_j\|u(t)\|_{q_j+1}^{q_j+1}.
\end{align*} 
We infer from $q_1\cdots<q_{l_0}<p_1<\cdots<p_r$ that 
\begin{align*}
&\sum_{i=1}^r\frac{a_i}{p_i+1}\|u(t)\|_{p_i+1}^{p_i+1}-\sum_{j=1}^{l}\frac{b_j}{q_j+1}\|u(t)\|_{q_j+1}^{q_j+1}\\
=&\sum_{i=1}^r\frac{a_i}{p_i+1}\|u(t)\|_{p_i+1}^{p_i+1}-\sum_{j=1}^{l_0}\frac{b_j}{q_j+1}\|u(t)\|_{q_j+1}^{q_j+1}-\sum_{j=l_0+1}^{l}\frac{b_j}{q_j+1}\|u(t)\|_{q_j+1}^{q_j+1} \\ 
\le& \frac{1}{p_1+1}\left(\sum_{i=1}^{r}a_i\|u(t)\|_{p_i+1}^{p_i+1}-\sum_{j=1}^{l_0}b_j\|u(t)\|_{q_j+1}^{q_j+1}\right)\\
\le& \frac{2}{p_1+1}\E(t),
\end{align*}
which implies that
\begin{align*}
E(t)=\E(t)-\sum_{i=1}^r\frac{a_i}{p_i+1}\|u(t)\|_{p_i+1}^{p_i+1}+\sum_{j=1}^{l}\frac{b_j}{q_j+1}\|u(t)\|_{q_j+1}^{q_j+1}\ge\frac{p_1-1}{p_1+1}\E(t)\ge 0.
\end{align*}
Combining the above inequality with \eqref{<d0}, we deduce that for any $t\in (0,T),$ $$\E(t)\le\frac{p_1+1}{p_1-1}d.$$ 
Hence, we infer from Lemma \ref{lem:1} that $u$ is a global solution of problem \eqref{equation}.
\end{proof}
\begin{remark}\label{rmk:44}
In the proof of Theorem \ref{globalex}, we have proved  that $u^t\in \W_1$ for any $t\ge0.$ A similar argument shows that if $E(0)<d$ and $u_0\in \W_2,$ then $u^t\in \W_2$ for any $t\in (0,T_{max}),$ where $T_{max}$ is defined as in \eqref{lifespan}.
\end{remark}
\begin{corollary}\label{cor:1}
Assume that Assumption \ref{Assumption1} is in force and let $u$ be the weak solution of problem \eqref{equation} on its maximal existence interval $(-\infty,T).$ If $E(0)<d$ and 
\begin{align*}
u_0(0)\in \mathcal{V}_1:=\left\{u\in H^1_0(\Omega):J(u)<d\ \text{and}\ \|\nabla u\|_2^2+\sum_{j=1}^{l_0}b_j\|u\|_{q_j+1}^{q_j+1}>\sum_{i=1}^ra_i\|u\|_{p_i+1}^{p_i+1}\right\}\bigcup\{0\}.
\end{align*}
Then the solution $u$ of problem \eqref{equation} is global, i.e., $T=+\infty$. Furthermore, we have
\begin{align*}
0\le E(t)\le E(0)<d\quad \text{and}\quad0\le \frac{p_1-1}{p_1+1}\E(t)\le E(t),\quad\text{for all}\ t\ge0.
\end{align*} 
\end{corollary}
\begin{proof}
Notice that the conditions that $u_0(0)\in \mathcal{V}_1$ and $E(0)<d$ imply that $u_0\in \mathcal{W}_1,$ then the conclusion can be directly derived from Theorem \ref{globalex}.
\end{proof}
\begin{remark}\label{rmk:3.9}
We conclude from Lemma \ref{lem:3.4} and its proof that $d=\inf\limits_{u\in \mathcal{N}}J(u).$ Thus, the different potential well can be defined by
\begin{align*}
\mathcal{V}:=\{u\in H^1_0(\Omega):J(u)<d\},
\end{align*}
such that we can prove a global existence result in the potential well $\mathcal{V}$(i.e., Corollary \ref{cor:1}) for the weak solutions of problem \eqref{equation} by carrying out the similar proof of Theorem \ref{globalex}. By comparison, the global existence theorem stated in Theorem \ref{globalex}, which takes full advantage of the entire history, is a better result and more suitable for our system.
	\end{remark}
	\section{Energy decay}\label{decay}
This section is devoted to the proof of the energy decay rates for weak solutions of problem \eqref{equation} in the potential well based on the framework in \cite{Guo2018energy,Irena1993uniform}.

For the sake of convenience, define $D(t)$ by
\begin{align*}
D(t)=\int_{0}^{t}\|u_t(\tau)\|_{m+1}^{m+1}\,d\tau-\frac{1}{2}\int_{0}^{t}\int_{0}^{+\infty}\mu^{\prime}(s)\|\nabla w^{\tau}(s)\|_2^2\,dsd\tau.
\end{align*}
First of all, we will establish some useful inequalities in the case that the relaxation kernel $\mu(s)$ is exponential decay (i.e. $\mu(s)=e^{-as}$ for some $a>0$) or polynomial decay (i.e. $\mu(s)=(1+bs)^{-\alpha}$ for some positive constants $b$ and $\alpha$).
\begin{proposition}\label{proposition}
Assume that Assumption \ref{Assumption1} is in force, $\E(0)\le y_0,$ $E(0)<d_0$ and $u_0\in \W_1,$ let $M_1:=\sup\limits_{\tau\in \R^+}\|u(\tau)\|_{m+1}<+\infty$ and let $u$ be the weak solution of problem \eqref{equation}. Then the following conclusions are true:
\begin{enumerate}[\rm{(}I\rm{)}]
\item If there exists a positive constant $a$ satisfying $\mu^{\prime}(s)+a\mu(s)\le 0$ for any $s\ge 0,$ then there exist positive constants $C_1,$ $T_1$ depending on $a$, $E(0)$ and $M_1,$ such that
\begin{align*}
E(t)\le C_1\left(D(t)^{\frac{2}{m+1}}+D(t)\right)
\end{align*}
for any $t\ge T_1.$
\item If there exist positive constants $b$ and $\theta\in (1,2)$ satisfying $\mu^{\prime}(s)+b\mu(s)^\theta\le 0$ for all $s\ge0,$ let $M_2:=\sup\limits_{\tau\in \R^-}\|\nabla u_0(\tau)\|_2<+\infty,$ then for any $\sigma\in (0,2-\theta)$, there exist positive constants $C_2$, $T_2$ depending on $E(0),$ $M_1$ and $C_3$ depending on $\theta,$ $b$, $\sigma,$ $M_2,$ $E(0),$ such that
\begin{align*}
E(t)\le C_2\left(D(t)^{\frac{2}{m+1}}+D(t)\right)+C_3D(t)^{\frac{\sigma}{\sigma+\theta-1}}
\end{align*}
for any $t\ge T_2.$	
\end{enumerate}	
\end{proposition}
\begin{proof}
It follows from Theorem \ref{globalex} that the weak solution $u$ of problem \eqref{equation} is global. We conclude from $M_1<+\infty$ that $u\in L^{m+1}(\Omega\times (0,t))$ for any $t\ge0.$ Thus, we can choose $\phi=u$ in equality \eqref{solutioneq} as the test function to obtain 
\begin{align}\label{eq:4.4}
\nonumber&\int_{\Omega}u_t(t)u(t)\,dx-\int_{\Omega}u_t(0)u(0)\,dx-\int_{0}^{t}\|u_t(\tau)\|_{2}^{2}\,d\tau+\int_{0}^{t}\|\nabla u(\tau)\|_{2}^{2} \, d\tau\\
\nonumber&+\int_{0}^{t}\int_{\Omega}|u_t(\tau)|^{m-1}u_t(\tau)u(\tau)\,dxd\tau+\int_{0}^{t}\int_{0}^{+\infty}\int_{\Omega}\nabla w^{\tau}(s)\cdot\nabla u(\tau) \mu(s)\,dxdsd\tau\\
=&\int_{0}^{t}\int_{\Omega} f(u(\tau))u(\tau)\,dxd\tau.
\end{align}
By using the definition of $\E(t),$ we obtain
\begin{align}\label{2E}
\nonumber 2\int_{0}^{t}\E(\tau)\,d\tau=&\int_{0}^{t}\|u_t(\tau)\|_2^2\,d\tau+\int_{0}^{t}\|\nabla u(\tau)\|_2^2\,d\tau+\int_{0}^{t}\int_{0}^{+\infty}\mu(s)\|\nabla w^{\tau}(s)\|_2^2\,dsd\tau\\
\nonumber=&-\int_{\Omega}u_t(t)u(t)\,dx+\int_{\Omega}u_t(0)u(0)\,dx+2\int_{0}^{t}\|u_t(\tau)\|_2^2\,d\tau\\
\nonumber&+\int_{0}^{t}\int_{0}^{\infty}\mu(s)\|\nabla w^{\tau}(s)\|_2^2\,dsd\tau-\int_{0}^{t} \int_{0}^{+\infty} \int_{\Omega} \nabla w^{\tau}(s)\cdot\nabla u(\tau)\mu(s)\,dx ds d\tau\\
&-\int_{0}^{t}\int_{\Omega}|u_t(\tau)|^{m-1}u_t(\tau)u(\tau)\,dxd\tau+\int_{0}^{t}\int_{\Omega} f(u(\tau))u(\tau)\,dxd\tau.
\end{align}
In the following, we will estimate each term of the right-hand side of \eqref{2E}. 
\begin{enumerate}[\rm{(}1\rm{)}]
\item The estimate of $\left|-\int_{\Omega}u_t(t)u(t)\,dx+\int_{\Omega}u_t(0)u(0)\,dx\right|.$
		
By applying Cauchy-Schwartz inequality and Sobolev inequality, we have
\begin{align*}
\left|\int_{\Omega}u_t(t)u(t)\,dx\right|\le \frac{1}{2}(\|u_t(t)\|_2^2+\|u(t)\|_2^2)\le \frac{C}{2}\left(\|u_t(t)\|_2^2+\|\nabla u(t)\|_2^2\right)\le C\E(t),
\end{align*}
which implies that
\begin{align}\label{I1}
\nonumber\left|-\int_{\Omega}u_t(t)u(t)\,dx+\int_{\Omega}u_t(0)u(0)\,dx\right|\le& C(\E(t)+\E(0))\\
\nonumber\le&C\left(\frac{p_1+1}{p_1-1}\right)\left(E(t)+E(0)\right)\\
\le&C\left(\frac{p_1+1}{p_1-1}\right)\left(2E(t)+D(t)\right),
\end{align}
where we have used Theorem	\ref{globalex} and the energy identity \eqref{Eenergyidentity}.
\item The estimate of $\int_{0}^{t}\|u_t(\tau)\|_2^2\,d\tau.$ 

We infer from H\"{o}lder's inequality that
\begin{align}\label{I2}
\nonumber\int_{0}^{t}\int_{\Omega}|u_t(\tau)|_2^2\,dxd\tau\le&\left(t|\Omega|\right)^{\frac{m-1}{m+1}}\left(\int_{0}^{t}\int_{\Omega}|u_t(\tau)|^{m+1}\,dxd\tau\right)^{\frac{2}{m+1}}\\
\le& \left(t|\Omega|\right)^{\frac{m-1}{m+1}}D(t)^{\frac{2}{m+1}}.
\end{align}

\item The estimate of $\int_{0}^{t}\int_{\Omega} f(u(\tau))u(\tau)\,dxd\tau.$
	
We conclude from $E(0)<d_0$ that there exists a constant $\delta>0,$ such that $E(0)=d_0-\delta.$ Thus, it follows from Sobolev inequality and \eqref{E'} that for any $\tau\geq 0,$
\begin{align*}
\|\nabla u(\tau)\|_2^2\le& 2E(\tau)+2\int_{\Omega}F(u(x,\tau))\,dx\\
\le& 2d_0-2\delta+\sum_{i=1}^{r}\frac{2a_i}{p_i+1}\gamma_i^{p_i+1}\|\nabla u(\tau)\|_2^{p_i+1},
\end{align*}
which entails that
\begin{align*}
G(y_0)-G\left(\frac{\|\nabla u(\tau)\|_2^2}{2}\right)\geq \delta
\end{align*}
for any $\tau\geq 0.$ Thus, there exists some $\delta_0>0,$ such that
\begin{align}\label{or}
\|\nabla u(\tau)\|_2^2>2y_0+2\delta_0\,\,\quad \text{or}\quad \|\nabla u(\tau)\|_2^2<2y_0-2\delta_0
\end{align}
for any $\tau\geq 0.$ Since
\begin{align*}
\frac{\|\nabla u(0)\|_2^2}{2}\le \E(0)\le y_0
\end{align*}
and $\|\nabla u(t)\|_2^2$ is continuous in $t,$ it follows from \eqref{or} that
\begin{align}\label{<2y_0}
\|\nabla u(\tau)\|_2^2<2y_0-2\delta_0
\end{align}
for any $\tau\geq 0.$ We conclude from \eqref{y0} and \eqref{<2y_0} that there exists a positive constant $\epsilon_0\in (0,1),$ such that
\begin{align}\label{epsil}
\sum_{i=1}^ra_i\gamma_i^{p_i+1}\|\nabla u(\tau)\|_2^{p_i-1}<\frac{1}{2}\sum_{i=1}^{r}a_i\left(2\gamma_i^2\right)^{\frac{p_i+1}{2}}y_0^{\frac{p_i-1}{2}}-\epsilon_0=1-\epsilon_0.
\end{align}
Therefore, we have
\begin{align}\label{I4}
\nonumber\int_{0}^{t}\int_{\Omega} f(u(\tau))u(\tau)\,dxd\tau=&\int_{0}^{t}\left(\sum_{i=1}^{r}a_i\|u(\tau)\|_{p_i+1}^{p_i+1}-\sum_{j=1}^{l}b_j\|u(\tau)\|_{q_j+1}^{q_j+1}\right)\,d\tau\\
\nonumber\le& \int_{0}^{t}\sum_{i=1}^{r}a_i\gamma_i^{p_i+1}\|\nabla u(\tau)\|_{2}^{p_i+1}\,d\tau-\int_{0}^{t}\sum_{j=1}^lb_j\|u(\tau)\|_{q_j+1}^{q_j+1}\,d\tau\\
\nonumber\le& (1-\epsilon_0)\int_{0}^{t}\|\nabla u(\tau)\|_2^2\,d\tau-\int_{0}^{t}\sum_{j=1}^lb_j\|u(\tau)\|_{q_j+1}^{q_j+1}\,d\tau\\
\le& (2-2\epsilon_0)\int_{0}^{t}\E(\tau)\,d\tau-\int_{0}^{t}\sum_{j=1}^lb_j\|u(\tau)\|_{q_j+1}^{q_j+1}\,d\tau.
\end{align}
\item The estimate of $\left|\int_{0}^{t}\int_{\Omega}|u_t(\tau)|^{m-1}u_t(\tau)u(\tau)\,dxd\tau\right|.$

It follows from H\"{o}lder's inequality that
\begin{align}\label{413}
\left|\int_{0}^{t}\int_{\Omega}|u_t(\tau)|^{m-1}u_t(\tau)u(\tau)\,dxd\tau\right|
\le \left(\int_{0}^{t}\int_{\Omega}|u|^{m+1}\,dxd\tau\right)^{\frac{1}{m+1}}\left(\int_{0}^{t}\int_{\Omega}|u_t|^{m+1}\,dxd\tau\right)^{\frac{m}{m+1}}.
\end{align}
If $1\le m\le 5$, we deduce from Sobolev inequality that
\begin{align}\label{416}
\|u\|^{m+1}_{L^{m+1}(\Omega\times (0,t))}\le CM_1^{m-1}\int_{0}^{t}\|\nabla u\|_2^{2}\,d\tau\le CM_1^{m-1}\int_{0}^{t}\E(\tau)\,d\tau.
\end{align}
Along with \eqref{413} and \eqref{416}, it yields 
\begin{align*}
\nonumber\left|\int_{0}^{t}\int_{\Omega}|u_t(\tau)|^{m-1}u_t(\tau)u(\tau)\,dxd\tau\right|\le& CD(t)^{\frac{m}{m+1}}\left(\int_{0}^{t}\E(\tau)\,d\tau\right)^{\frac{1}{m+1}}\\
\le&\frac{\epsilon_0}{2}\int_{0}^{t}\E(\tau)\,d\tau+C_{\epsilon_0}D(t).
\end{align*}
In the case that $m>5,$ it follows from \eqref{413} and Young's inequality that
\begin{align*}
\left|\int_{0}^{t}\int_{\Omega}|u_t(\tau)|^{m-1}u_t(\tau)u(\tau)\,dxd\tau\right|
\le& t^{\frac{1}{m+1}}M_1D(t)^{\frac{m}{m+1}}\\
=&t^{\frac{1}{m+1}}M_1D(t)^{\frac{m-2}{m-1}}D(t)^{\frac{2}{(m-1)(m+1)}}\\
\le& \frac{(m-2)M_1^{\frac{m-1}{m-2}}D(t)+t^{\frac{m-1}{m+1}}D(t)^{\frac{2}{m+1}}}{m-1}.
\end{align*}
Hence, we conclude that
\begin{align}\label{I5}
\nonumber&\left|\int_{0}^{t}\int_{\Omega}|u_t(\tau)|^{m-1}u_t(\tau)u(\tau)\,dxd\tau\right|\\
\le& \frac{\epsilon_0}{2}\int_{0}^{t}\E(\tau)\,d\tau+t^{\frac{m-1}{m+1}}D(t)^{\frac{2}{m+1}}+C_{\epsilon_0,M_1}D(t).
\end{align}

\item The estimate of $\left|\int_{0}^{t} \int_{0}^{+\infty}\int_{\Omega}\nabla w^{\tau}(s)\cdot\nabla u(\tau) \mu(s)\,dsd\tau\right|.$

By using Cauchy-Schwartz inequality and Young's inequality, we obtain
\begin{align}\label{I3}
\nonumber&\left|\int_{0}^{t}\int_{0}^{+\infty}\int_{\Omega} \mu(s)\nabla w^{\tau}(s)\cdot\nabla u(\tau) \, dsd\tau\right|\\
\nonumber\le&\frac{\epsilon_0}{4}\int_{0}^{t}\|\nabla u(\tau)\|_{2}^2\,d\tau+C_{\epsilon_0}\int_{0}^{t}\int_{0}^{+\infty}\mu(s)\|\nabla w^{\tau}(s)\|_2^2\,dsd\tau\\
\le&\frac{\epsilon_0}{2}\int_{0}^{t}\E(\tau)\,d\tau+C_{\epsilon_0}\int_{0}^{t}\int_{0}^{+\infty}\mu(s)\|\nabla w^{\tau}(s)\|_2^2\,dsd\tau.
\end{align}
We conclude from inequalities \eqref{2E}, \eqref{I1}, \eqref{I2}, \eqref{I5}, \eqref{I3} and \eqref{I6} that there exists a positive constant $C$ depending on $M_1$ and $E(0),$ such that for any $t\ge1,$
\begin{align}\label{eq:qq}
\nonumber\int_{0}^{t}\E(\tau)\,d\tau\le &C\left(E(t)+tD(t)^{\frac{2}{m+1}}+\int_{0}^{t}\int_{0}^{+\infty}\mu(s)\|\nabla w^{\tau}(s)\|_2^2\,dsd\tau+D(t)\right)\\
&-\frac{1}{\epsilon_0}\int_{0}^{t}\sum_{j=1}^lb_j\|u(\tau)\|_{q_j+1}^{q_j+1}\,d\tau.
\end{align}
It follows from \eqref{E'}, \eqref{eq:qq} and the definition of $E(t)$ that there exists a positive constant $\mathcal{K}_1$ depending on $M_1$ and $E(0),$ such that for any $t\geq 1,$
\begin{align}\label{la}
\nonumber tE(t)\le&\int_{0}^{t}E(\tau)\,d\tau\le\int_{0}^{t}\left(\E(\tau)+\sum_{j=1}^{l}\frac{b_j}{q_j+1}\|u(\tau)\|_{q_j+1}^{q_j+1}\right)\,d\tau\\
\le&\mathcal{K}_1\left(E(t)+tD(t)^{\frac{2}{m+1}}+\int_{0}^{t}\int_{0}^{+\infty}\mu(s)\|\nabla w^{\tau}(s)\|_2^2\,dsd\tau+D(t)\right).
\end{align}
\item The estimate of $\int_{0}^{t}\int_{0}^{+\infty}\mu(s)\|\nabla w^{\tau}(s)\|_2^2\,dsd\tau.$ In what follows, we will give the estimate of $\int_{0}^{t}\int_{0}^{+\infty}\mu(s)\|\nabla w^{\tau}(s)\|_2^2\,dsd\tau$ under two different assumptions on the relaxation kernel $\mu(s).$ 
\begin{enumerate}[\rm{(}I\rm{)}]
\item If $\mu^{\prime}(s)+a\mu(s)\le 0,$ we have
\begin{align}\label{I6}
\nonumber\int_{0}^{t}\int_{0}^{+\infty}\mu(s)\|\nabla w^{\tau}(s)\|_2^2\,dsd\tau\le &-\frac{1}{a}\int_{0}^{t}\int_{0}^{+\infty}\mu^{\prime}(s)\|\nabla w^{\tau}(s)\|_2^2\,dsd\tau\\
\le& \frac{1}{a}D(t).
\end{align}
Thus, we conclude from inequalities \eqref{la}-\eqref{I6}  that for any $t\geq 2\mathcal{K}_1,$
\begin{align*}
E(t)\le 2\mathcal{K}_1D(t)^{\frac{2}{m+1}}+(1+\frac{1}{a})D(t).
\end{align*}
\item If $\mu^{\prime}(s)+b\mu(s)^\theta\le 0$ with $1<\theta<2,$ we deduce from H\"{o}lder's inequality that
\begin{align}\label{r}
\nonumber&\int_{0}^{t}\int_{0}^{+\infty}\mu(s)\|\nabla w^{\tau}(s)\|_2^2\,ds\\
\nonumber=&\int_{0}^{t}\int_{0}^{+\infty}\left(\mu(s)^{\frac{\sigma \theta}{\sigma+\theta-1}}\|\nabla w^{\tau}(s)\|_2^{\frac{2\sigma}{\sigma+\theta-1}}\right)\left(\mu(s)^{\frac{(1-\sigma)(\theta-1)}{\sigma +\theta-1}}\|\nabla w^{\tau}(s)\|_2^{\frac{2(\theta-1)}{\sigma+\theta-1}}\right)\,dsd\tau\\
\le&\left(\int_{0}^{t}\int_{0}^{+\infty}\mu(s)^\theta\|\nabla w^{\tau}(s)\|_2^2\,dsd\tau\right)^{\frac{\sigma}{\sigma+\theta-1}}\left(\int_{0}^{t}\int_{0}^{+\infty}\mu(s)^{1-\sigma}\|\nabla w^{\tau}(s)\|_2^2\,dsd\tau\right)^{\frac{\theta-1}{\sigma+\theta-1}}.
\end{align}
Since $\mu^{\prime}(s)+b\mu(s)^\theta\leq 0$ and
\begin{align*}
\|\nabla u(\tau)\|_2^2\le \max\left\{\sup_{\tau\le 0}\|\nabla u(\tau)\|_2^2,\sup_{\tau\ge 0}\|\nabla u(\tau)\|_2^2\right\}\le \max\left\{M_2,\frac{2(p_1+1)}{p_1-1}E(0)\right\}
\end{align*}
for any $\tau\in\R,$ we have 
\begin{align}\label{1-sigma}
\nonumber\int_{0}^{t}\int_{0}^{+\infty}\mu(s)^{1-\sigma}\|\nabla w^\tau(s)\|_2^2\,dsd\tau
\le&4t\sup_{\tau\in \R}\|\nabla u(\tau)\|_2^2\int_{0}^{+\infty}(\mu(0)^{1-\theta}+b(\theta-1)s)^{-\frac{1-\sigma}{\theta-1}}\,ds\\
\le &\frac{4(\theta-1)t}{(2-\sigma-\theta)}\mu(0)^{2-\sigma-\theta}\sup_{\tau\in \R}\|\nabla u(\tau)\|_2^2,
\end{align}
which entails that there exists a positive constant $\mathcal{K}_2$ depending on $\mu(0),$ $\sigma,$ $\theta,$ $M_2$, $b$ and $E(0),$ such that for any $t\geq 1,$
\begin{align}\label{I62}
\nonumber\int_{0}^{t}\int_{0}^{+\infty}\mu(s)\|\nabla w^\tau(s)\|_2^2\,ds\le& \mathcal{K}_2t^{\frac{\theta-1}{\sigma+\theta-1}}\left(-\int_{0}^{t}\int_{0}^{\infty}\|\nabla w^\tau(s)\|_2^2\mu^{\prime}(s)\,dsd\tau\right)^{\frac{\sigma}{\sigma+\theta-1}}\\
\le &\mathcal{K}_2tD(t)^{\frac{\sigma}{\sigma+\theta-1}}.
\end{align}
We conclude from inequality \eqref{la} and inequality \eqref{I62} that for any $t\geq 2\mathcal{K}_1,$
\begin{align*}
E(t)\le 2\mathcal{K}_1D(t)^{\frac{2}{m+1}}+D(t)+2\mathcal{K}_1\mathcal{K}_2D(t)^{\frac{\sigma}{\sigma+\theta-1}}.
\end{align*}
\end{enumerate}
\end{enumerate}			
\end{proof}
\begin{remark}
Notice that if $m\le 5$, it follows from Sobolev inequality that for any $\tau\geq 0,$
\begin{align*}
\|u(\tau)\|^2_{m+1}\le C\|\nabla u(\tau)\|_2^2\le 2C\E(\tau)\le \frac{2(p_1+1)}{p_1-1}CE(\tau)\le \frac{2(p_1+1)}{p_1-1}CE(0).
\end{align*}
Consequently, $M_1$ can be bounded by $E(0)$ for $m\le 5,$ such that the finiteness of $M_1$ follows naturally. At this point, $C_1$ and $T_1$ depend only on $a$ and $E(0)$, while $C_2$ and $T_2$ depend only on $E(0)$.
\end{remark}
\begin{remark}
For our proof, it suffices to assume that $\|\nabla u_0(0)\|_2 \le 2y_0$ instead of $\E(0) \le y_0$. On the other hand, we can replace the condition $\E(0)\le y_0$ by $E(0)\le \frac{(p_r+1)(p_1-1)}{(p_r-1)(p_1+1)}d_0$. Indeed, we infer from \eqref{y0} and \eqref{d00} that
\begin{align*}
d_0\le \frac{p_r-1}{p_r+1}y_0.
\end{align*}
If $E(0)\le \frac{(p_r+1)(p_1-1)}{(p_r-1)(p_1+1)}d_0,$ then it follows from \eqref{Ecompar} that
\begin{align*}
\E(0)\le \frac{p_1+1}{p_1-1}E(0)\le \frac{p_r+1}{p_r-1}d_0\le y_0.
\end{align*}
\end{remark}
\begin{remark}
The constants $C_1,\ C_2,\ T_1$ and $T_2$ will tend to $+\infty$ as $E(0)$ tends to $d_0$; while $C_3$ will diverge to $+\infty$ as $\sigma+\theta\to2.$ 
\end{remark}
\begin{theorem}\label{4.5}
Assume that Assumption \ref{Assumption1} is in force, $\E(0)\le y_0,$ $E(0)<d_0,$ $u_0\in \W_1,$ let $M_1:=\sup\limits_{\tau\in \R^+}\|u(\tau)\|_{m+1}<+\infty$ and let $u$ be the weak solution of problem \eqref{equation}. Then the following conclusions are true:
\begin{enumerate}[\rm{(}I\rm{)}]
\item Assume that there exists a positive constant $a_1$ satisfying $\mu^{\prime}(s)+a\mu(s)\le 0$ for any $s\ge0,$ then we have
\begin{enumerate}[\rm{(}1\rm{)}]
\item if $m=1,$ we obtain
\begin{align*} 
E(t)\le K_1e^{-\kappa t}
\end{align*}
for any $t\ge 0,$ where $K_1,$ $\kappa$ are positive constants depending on $a$, $E(0)$ and $M_1.$
\item if $m>1,$ we have
\begin{align*} 
E(t)\le K_2\left(1+t\right)^{\frac{2}{1-m}}
\end{align*}
for any $t\ge 0,$ where $K_2>0$ is a constant depending on $a$, $M_1$ and $E(0).$
\end{enumerate}
\item Assume that there exists a positive constant $a_2$ satisfying $\mu^{\prime}(s)+b\mu(s)^\theta\le 0$ for any $s\ge 0$ with $\theta\in (1,2),$ let $M_2:=\sup\limits_{t\in \R^-}\|\nabla u_0(\tau)\|_2<+\infty$ and $\sigma\in (0,2-\theta),$ then we have
\begin{enumerate}[\rm{(}1\rm{)}]
\item if $m=1,$ we obtain
\begin{align*} 
E(t)\le K_3(1+t)^{-\frac{\sigma}{\theta-1}}
\end{align*}
for any $t\ge 0,$ where $K_3>0$ is a constant depending on $\theta,$ $b,$ $\sigma,$ $M_1,$ $M_2$ and $E(0).$
\item if $m>1,$ we obtain 
\begin{align*} 
E(t)\le K_4(1+t)^{-\min\left\{\frac{\sigma}{\theta-1},\frac{2}{m-1}\right\}}
\end{align*}
for any $t\ge 0$, where $K_4>0$ is a constant depending on $\theta,$ $b,$ $\sigma,$ $M_1,$ $M_2$ and $E(0).$
\end{enumerate}
\end{enumerate}  		
\end{theorem}
\begin{proof}
\begin{enumerate}[\rm{(}I\rm{)}]
\item If $\mu^{\prime}(s)+a\mu(s)\le 0,$ define $\Phi(s):=C_1\left(s^{\frac{2}{m+1}}+s\right),$ then for sufficiently large $t>0$, we deduce from the energy identity \eqref{energyidentity} and Proposition \ref{proposition} that
\begin{align}\label{4.30}
E(t)\le \Phi(D(t))=\Phi(E(0)-E(t)).
\end{align}
Due to the monotone increasing property of $\Phi$ and inequality \eqref{4.30}, we can choose a sufficiently large $T$ depending on $E(0),$ such that
\begin{align}\label{4.31}
\left(I+\Phi^{-1}\right)E(T)\le E(0).
\end{align}
Since $E(t)$ is non-increasing, we similarly conclude from Proposition \ref{proposition} that for any integer $n\geq 0,$
\begin{align}\label{4.34}
\left(I+\Phi^{-1}\right)E((n+1)T)\le E(nT)
\end{align}
by considering problem \eqref{equation} with $u_0(t)=u(t+nT).$

In what follows, we will show that $E(t)$ is bounded by $S(t)=|Q(t)|,$ where $Q(t)$ is the solution of the following ordinary differential equation:
\begin{equation}\label{S(t)eq}
\begin{cases}
|Q(t)|^{\prime}+(I+\Phi)^{-1}|Q(t)|=0,\\
|Q(0)|=E(0).
\end{cases}
\end{equation}
Since $(I+\Phi)^{-1}$ has a bounded derivative on $[0,+\infty)$, we conclude from Picard-Lindel\"{o}f Theorem that problem \eqref{S(t)eq} admits a unique global solution on $\R^+.$ Since $S(t)\ge0$ for any $t\geq 0,$ which entails that $S'(t)\le0.$ Consequently, there exists $S_0\geq 0,$ such that $\lim\limits_{t\to \infty}S(t)=S_0.$ Moreover, we deduce from \eqref{S(t)eq} and $S(t)\ge S_0$ for any $t\geq 0$ that
\begin{align*}
S(t)=E(0)-\int_{0}^{t}(I+\Phi)^{-1}S(\tau)\,d\tau\le E(0)-t(I+\Phi)^{-1}S_0.
\end{align*}
Let $t\to+\infty$ in the above inequality, we conclude from $S(t)>0$ for any $t\geq 0$ that $S_0=0.$ In what follows, we will prove the following conclusion by induction: 
\begin{align*}
E(nT)\le S(n),\quad \forall\,\,\, n=0,1,2,\cdots.
\end{align*}
Assume that there exists a integer $n\geq 0,$ such that $E(nT)\le S(n).$ Since
\begin{align*}
S(n+1)=S(n)-\int_{n}^{n+1}(I+\Phi)^{-1}S(t)\,dt\ge S(n)-(I+\Phi)^{-1}S(n)
\end{align*}
and 
\begin{align*}
I-(I+\Phi)^{-1}&=(I+\Phi)\circ(I+\Phi)^{-1}-(I+\Phi)^{-1}=\Phi\circ(I+\Phi)^{-1}=\Phi\circ(\Phi^{-1}\circ\Phi+\Phi)^{-1}\\&=\Phi\circ\left((\Phi^{-1}+I)\circ\Phi\right)^{-1}=\Phi\circ\Phi^{-1}\circ(\Phi^{-1}+I)^{-1}=(I+\Phi^{-1})^{-1},
\end{align*}
we obtain
\begin{align*}
S(n+1)\ge \left(I-(I+\Phi)^{-1}\right)S(n)=(I+\Phi^{-1})^{-1}S(n),
\end{align*}
which entails that
\begin{align*}
E((n+1)T)&\le (I+\Phi^{-1})^{-1}E(nT) \\
&\le (I+\Phi^{-1})^{-1}S(n)\\
&\le S(n+1).
\end{align*}
Notice that both $E(t)$ and $S(t)$ are monotone decreasing in time, we obtain
\begin{align}\label{S}
E(t)\le S\left(\frac{t}{T}-1\right)
\end{align}
for any $t\ge T.$

\begin{enumerate}[\rm{(}1\rm{)}]
\item If $m=1,$ then there exists a positive constant $\kappa_0$ depending on $a$, $E(0)$ and $M_1,$ such that
\begin{align*}
S(t)=E(0)e^{-\kappa_0 t}
\end{align*} 
for any $t\geq 0.$ Thus, it follows from \eqref{S} that there exist two positive constants $K_1$ and $\kappa$ depending on $E(0)$ and $M_1,$ such that
\begin{align*}
E(t)\le K_1e^{-\kappa t}
\end{align*}
for any $t\geq 0.$

\item If $m>1,$ we conclude from $\lim\limits_{t\to \infty}S(t)=0$ and \eqref{S(t)eq} that there exists some time $T_0,$ such that
\begin{align*}
S^{\prime}(t)+CS(t)^{\frac{m+1}{2}}\le 0
\end{align*}
for any $t\geq T_0,$ which entails that there exists a positive constant $C$ depending on $a$, $E(0)$ and $M_1,$ such that 
\begin{align*}
S(t)\le C(2+t)^{\frac{2}{1-m}}
\end{align*}
for any $t\geq T_0.$ It follows from \eqref{S} that there exists a positive constant $K_2$ depending on $a$, $E(0)$ and $M_1,$ such that
\begin{align*}
E(t)\le K_2(1+t)^{\frac{2}{1-m}}
\end{align*}
for any $t\geq 0.$
\end{enumerate}		
\item If $\mu^{\prime}(s)+b\mu(s)^\theta\le 0,$ denote by $\Psi(s)=C_2\left(s^{\frac{2}{m+1}}+s\right)+C_3s^{\frac{\sigma}{\sigma+\theta-1}},$ where the constants $C_2,$ $C_3$ are the same as in Proposition \ref{proposition}, then we have
\begin{align*}
E(t)\le \Psi(D(t))=\Psi(E(0)-E(t)).
\end{align*} 
By carrying out the above similar arguments, we immediately obtain the following results:
\begin{enumerate}[\rm{(}1\rm{)}]
\item If $m=1,$ then there exists a positive constant $K_3$ depending on $\theta,$ $b,$ $\sigma,$ $M_1,$ $M_2$ and $E(0),$ such that 
\begin{align*}
E(t)\le K_3(1+t)^{-\frac{\sigma}{\theta-1}}
\end{align*}
for any $t\geq 0.$
\item If $m>1,$ then there exists a positive constant $K_4$ depending on $\theta,$ $b,$ $\sigma,$ $M_1,$ $M_2$ and $E(0),$ such that 
\begin{align*}
E(t)\le K_4(1+t)^{{-\min\{\frac{\sigma}{\theta-1},\frac{2}{m-1}\}}}
\end{align*}
for any $t\geq 0.$
\end{enumerate}
\end{enumerate}
\end{proof}	
\section{Blow-up}\label{blo}
In this section, we will analyze blow-up of weak solutions for problem \eqref{equation} under some different assumptions on initial data based on the ideas in \cite{Guo2017blowup,Georgiev1994existence,Sun2019blow}.
	\subsection{Blow-up of solutions with negative initial energy}
	This subsection is devoted to proving the finite-time blow-up of weak solutions for problem \eqref{equation} in the case of negative initial energy, i.e., $E(0)<0$.
\begin{theorem}\label{negative}
Assume that Assumption \ref{Assumption1} is true and $E(0)<0.$ If $p_1>\max\{\sqrt{k(0)},q_l\}$ and $p_r>m$, then the weak solution $u$ of problem \eqref{equation} blows up in finite time. More precisely, there exists some time $T_{max}\in (0,+\infty),$ such that
\begin{align*}
\varlimsup_{t\rightarrow T^-_{max}}\|\nabla u(t)\|_2=+\infty.
\end{align*}   
\end{theorem}
\begin{proof}
Let $u$ be the weak solution to problem \eqref{equation} and $T_{max}$ is the maximal lifespan of $u$ defined as in \eqref{lifespan}. If $T_{max}<+\infty,$ we infer from Lemma \ref{lem:1} that 
\begin{align}\label{46}
\varlimsup_{t\to T_{max}^-}\E(t)=+\infty.
\end{align}
It follows from the definition of $E(t)$ and \eqref{E'} that 
\begin{align}\label{47}
\int_{\Omega}F(u(t))\,dx=\E(t)-E(t)\ge \E(t)-E(0).
\end{align} 
Along with \eqref{46}-\eqref{47}, we obtain
\begin{align}\label{48}
\varlimsup_{t\to T_{max}^-}\int_{\Omega}F(u(t))\,dx=+\infty.
\end{align}
On the other hand, it follows from Sobolev inequalities that there exists a positive constant $C,$ such that
\begin{align*}
\int_{\Omega}F(u(t))\,dx=\sum_{i=1}^{r}\frac{a_i}{p_i+1}\|u(t)\|_{p_i+1}^{p_i+1}-\sum_{j=1}^{l}\frac{b_j}{q_j+1}\|u(t)\|_{q_j+1}^{q_j+1}\le C\sum_{i=1}^{r}\|\nabla u(t)\|_2^{p_i+1},
\end{align*}
which entails that
\begin{align*}
\varlimsup_{t\rightarrow T^-_{max}}\|\nabla u(t)\|_2=+\infty.
\end{align*}
In what follows, we will prove that $T_{max}<+\infty.$ To do this, we first define $H(t)=-E(t)$ and $N(t)=\frac{1}{2}\|u(t)\|^2_2,$ we would like to prove that
\begin{align*}
Y(t):=H(t)^{1-\alpha}+\epsilon N^{\prime}(t)
\end{align*}
blows up in finite time, where $\alpha\in (0,1)$ and $\epsilon>0$ will be selected later. It follows from the fact that $m<p_r\le5$ and the definition of weak solution that $u\in C([0,T];H^1_0(\Omega))\subset L^{m+1}(\Omega\times (0,T))$ for any $0<T<T_{max}.$ Thus, by replacing $\phi$ by $u$ in \eqref{weaksolution}, we obtain
\begin{align*}
N^{\prime}(t)=&\int_{\Omega}u(t)u_t(t)\,dx\\
=&\int_{\Omega}u_t(0)u(0)\,dx+\int_0^t \|u_t(\tau)\|_2^2\,d\tau
-k(0)\int_0^t\|\nabla u(\tau)\|_2^2\,d\tau\\&-\int_0^t\int_0^{+\infty}\int_{\Omega}\nabla u(\tau - s)\cdot\nabla u(\tau)k'(s)\,dxdsd\tau\\
&-\int_0^t\int_{\Omega}|u_t(\tau)|^{m - 1}u_t(\tau)u(\tau)\,dxd\tau+\int_0^t\int_{\Omega}f(u(\tau))u(\tau)\,dxd\tau\\
=&\int_{\Omega}u_t(0)u(0)\,dx+\int_0^t \|u_t(\tau)\|_2^2\,d\tau-\int_0^t\|\nabla u(\tau)\|_2^2\,d\tau\\
&-\int_0^t\int_0^{+\infty}\int_{\Omega}\nabla w^\tau(s)\cdot\nabla u(\tau)\mu(s)\,dxdsd\tau\\
&-\int_0^t\int_{\Omega}|u_t(\tau)|^{m - 1}u_t(\tau)u(\tau)\,dxd\tau+\int_0^t\int_{\Omega}f(u(\tau))u(\tau)\,dxd\tau
\end{align*} 
for any $t\in [0,T_{max}).$ It is easy to verify that $N'(t)$ is absolutely continuous, then we can differentiate $N'(t)$ to obtain that
\begin{align}\label{twoorderdif}
\nonumber N''(t)=&\|u_t(t)\|^2_2-\|\nabla u(t)\|^2_2-\int_{0}^{+\infty}\mu(s)\int_{\Omega} \nabla w^t(s)\cdot\nabla u(t)\,dxds\\
&-\int_{\Omega}|u_t(t)|^{m-1}u_t(t)u(t)\,dx+\sum_{i=1}^{r}a_i\|u(t)\|_{p_i+1}^{p_i+1}-\sum_{j=1}^{l}b_j\|u(t)\|_{q_j+1}^{q_j+1},
\end{align}
which implies that $Y(t)$ is differentiable and 
\begin{align}\label{Yprime}
Y^{\prime}(t)=(1-\alpha)H(t)^{-\alpha}H^{\prime}(t)+\epsilon N^{\prime\prime}(t).
\end{align}
In the following, we would like to find an appropriate lower bound of right-hand side of \eqref{Yprime}. By using Cauchy-Schwartz inequality and Young's inequality, we conclude that for any $\delta>0,$
\begin{align}\label{w}
\left|\int_{0}^{+\infty}\mu(s)\int_{\Omega} \nabla w^t(s)\cdot\nabla u(t)\,dxds\right|
\le\delta\int_{0}^{+\infty}\mu(s)\|\nabla w^t(s)\|_2^2\,ds+\frac{k(0)-1}{4\delta}\|\nabla u\|_2^2. 
\end{align}
We deduce from H\"{o}lder's inequality and $p_r>m$ that
\begin{align}
\left|\int_{\Omega}u|u_t|^{m-1}u_t\,dx\right|\le \int_{\Omega}|u||u_t|^m\,dx\le \|u\|_{m+1}\|u_t\|^m_{m+1}\le C\|u\|_{p_r+1}\|u_t\|^m_{m+1}.
\end{align}
Since
\begin{align}\label{G't}
H^{\prime}(t)=-E^{\prime}(t)=\|u_t(t)\|^{m+1}_{m+1}-\frac{1}{2}\int_{0}^{+\infty}\mu'(s)\|\nabla w^t(s)\|_2^2\,ds\ge 0,
\end{align}
we conclude that $H(t)$ is non-decreasing for $t\in [0,T_{max})$ and
\begin{align*}
H(t)=&-E(t)=-\E(t)+\sum_{i=1}^{r}\frac{a_i}{p_i+1}\|u(t)\|_{p_i+1}^{p_i+1}-\sum_{j=1}^{l}\frac{b_j}{q_j+1}\|u(t)\|^{q_j+1}_{q_j+1}\\
\le&\sum_{i=1}^{r}\frac{a_i}{p_i+1}\|u(t)\|_{p_i+1}^{p_i+1}\\
\le&C\left(\sum_{i=1}^{r}\frac{a_i}{p_i+1}\|u(t)\|_{p_r+1}^{p_i+1}\right)\\
\le&C\max_{1\le i\le r}\left\{\|u(t)\|^{p_i+1}_{p_r+1}\right\},
\end{align*}
which implies that there exists a positive constant $c,$ such that
\begin{align}
H(t)^{-1}\ge c\min_{1\le i\le r}\left\{\|u(t)\|_{p_r+1}^{-p_i-1}\right\}.
\end{align}
Thus, we obtain 
\begin{align}\label{2.4}
\max_{1\le i\le r}\left\{H(t)^{-\frac{1}{p_i+1}}\right\}\ge c\|u(t)\|^{-1}_{p_r+1},
\end{align}
which entails that for any $\lambda>0,$
\begin{align*}
\left|\int_{\Omega}u|u_t|^{m-1}u_t\,dx\right|\le  &C\|u\|_{p_r+1}\|u_t\|_{m+1}^m\\
\le &C\|u\|_{p_r+1}^{\frac{p_r+1}{m+1}}\|u_t\|_{m+1}^m\|u\|_{p_r+1}^{\frac{m-p_r}{m+1}}\\
\le&\left(\lambda \|u\|_{p_r+1}^{p_r+1}+C_{\lambda}\|u_t\|_{m+1}^{m+1}\right)\max_{1\le i\le r}\left\{H(t)^{\frac{m-p_r}{(p_i+1)(m+1)}}\right\}.
\end{align*}
Let 
\begin{align*}
0<\alpha<\min\left\{\frac{p_r-m}{(p_r+1)(m+1)},\frac{p_r-1}{2(p_r+1)}\right\}<\frac{1}{2},
\end{align*}
we infer from \eqref{G't} that
\begin{align}\label{ut}
\nonumber\left|\int_{\Omega}u|u_t|^{m-1}u_t\,dx\right|\le &\lambda \max_{1\le i\le r}\left\{H(0)^{\frac{m-p_r}{(p_i+1)(m+1)}}\right\}\|u\|^{p_r+1}_{p_r+1}\\
&+C_{\lambda}\max_{1\le i\le r}\left\{H(0)^{\frac{m-p_r}{(p_i+1)(m+1)}+\alpha}\right\}H(t)^{-\alpha}H^{\prime}(t).
\end{align}
We conclude from \eqref{twoorderdif}, \eqref{w} and \eqref{ut} that for any $\delta,$ $\lambda>0,$
\begin{align}\label{2.11}
\nonumber N^{\prime\prime}(t)\ge &\|u_t(t)\|_2^2-\left(\frac{k(0)-1}{4\delta}+1\right)\|\nabla u(t)\|_2^2-\delta\int_{0}^{+\infty}\mu(s)\|\nabla w^t(s)\|_2^2\,ds\\
\nonumber&+\sum_{i=1}^{r-1}a_i\|u(t)\|_{p_i+1}^{p_i+1}+\left(a_r-\lambda\max_{1\le i\le r}\left\{H(0)^{\frac{m-p_r}{(p_i+1)(m+1)}}\right\}\right)\|u(t)\|_{p_r+1}^{p_r+1}\\
&-\sum_{j=1}^{l}b_j\|u(t)\|_{q_j+1}^{q_j+1}-C_{\lambda}\max_{1\le i\le r}\left\{H(0)^{\frac{m-p_r}{(p_i+1)(m+1)}+\alpha}\right\}H(t)^{-\alpha}H^{\prime}(t)
\end{align}
for any $t\in [0,T_{max}).$

We deduce from the definition of $E(t)$ and $H(t)=-E(t)$ that
\begin{align}\label{ww}
\int_{0}^{+\infty}\mu(s)\|\nabla w^t(s)\|_2^2\,ds=&-2H(t)-\|u_t(t)\|_2^2-\|\nabla u(t)\|_2^2\nonumber\\
&+\sum_{i=1}^{r}\frac{2a_i}{p_i+1}\|u(t)\|_{p_i+1}^{p_i+1}-\sum_{j=1}^{l}\frac{2b_j}{q_j+1}\|u(t)\|_{q_j+1}^{q_j+1},
\end{align}
which entails that
\begin{align}\label{2.13}
\nonumber N^{\prime\prime}(t)\ge &(1+\delta)\|u_t(t)\|_2^2+2\delta H(t)+\left(\delta-\frac{k(0)-1}{4\delta}-1\right)\|\nabla u(t)\|_2^2\\
\nonumber&+\sum_{i=1}^{r-1}\frac{a_i(p_i+1-2\delta)}{p_i+1}\|u(t)\|_{p_i+1}^{p_i+1}-\sum_{j=1}^{l}\frac{b_j(q_j+1-2\delta)}{q_j+1}\|u(t)\|_{q_j+1}^{q_j+1}\\
\nonumber&+\left(a_r-\lambda\max_{1\le i\le r}\left\{H(0)^{\frac{m-p_r}{(p_i+1)(m+1)}}\right\}-\frac{2a_r\delta}{p_r+1}\right)\|u(t)\|_{p_r+1}^{p_r+1}\\&-C_{\lambda}\max_{1\le i\le r}\left\{H(0)^{\frac{m-p_r}{(p_i+1)(m+1)}+\alpha}\right\}H(t)^{-\alpha}H^{\prime}(t)
\end{align}
for any $t\in [0,T_{max}).$ Let
\begin{align*}
\delta=\frac{\max\{\sqrt{k(0)},q_l\}+1}{2}
\end{align*}
and let $\lambda$ be such that
\begin{align*}
\lambda\max_{1\le i\le r}\left\{H(0)^{\frac{m-p_r}{(p_i+1)(m+1)}}\right\}=\frac{a_r(p_r+1-2\delta)}{2(p_r+1)},
\end{align*}
then we conclude from the assumption $\sqrt{k(0)}<p_1$ and $q_l<p_1$ that
\begin{align*}
\delta-\frac{k(0)-1}{4\delta}-1\ge0,\,\,p_i+1-2\delta>0,\,\,q_j+1-2\delta\le 0
\end{align*}
for any $i=1,\cdots ,r$ and any $j=1,\cdots,l,$ which entails that there exists a positive constant $c_0,$ such that
\begin{align*}
N^{\prime\prime}(t)\ge c_0\left(\|u_t(t)\|_2^2+H(t)+\|u(t)\|_{p_r+1}^{p_r+1}\right)-C_{\lambda}\max_{1\le i\le r}\left\{H(0)^{\frac{m-p_r}{(p_i+1)(m+1)}+\alpha}\right\}H(t)^{-\alpha}H^{\prime}(t).
\end{align*}
Furthermore, let $\epsilon>0$ be sufficiently small, such that
\begin{align*}
\epsilon \le \min\left\{\frac{1-\alpha}{C_{\lambda}\max\limits_{1\le i\le r}\left\{H(0)^{\frac{m-p_r}{(p_i+1)(m+1)}+\alpha}\right\}},1\right\},
\end{align*}
then we conclude that there exists a positive constant $c$ depending on $\epsilon,$ $p_i,$ $q_j$ and $k(0),$ such that
\begin{align}\label{Y'}
Y^{\prime}(t)\ge c\left(\|u_t(t)\|_2^2+H(t)+\|u(t)\|_{p_r+1}^{p_r+1}\right)>0
\end{align}
for any $t\in [0,T_{max}).$  
		
If $N^{\prime}(0)<0,$ we shall impose an extra restriction on $\epsilon:$
\begin{align*}
0<\epsilon\le -\frac{H(0)^{1-\alpha}}{2N^{\prime}(0)},
\end{align*}
then we infer from inequality \eqref{Y'} that
\begin{align*}
Y(t)\ge Y(0)\ge \frac{1}{2}H(0)^{1-\alpha}>0
\end{align*} 
for $t\in [0,T_{max}).$ Since
\begin{align}\label{Y}
Y(t)^{\frac{1}{1-\alpha}}\le 2^{\frac{1}{1-\alpha}}\left(H(t)+|N^{\prime}(t)|^{\frac{1}{1-\alpha}}\right)
\end{align}
for any $t\in [0,T_{max})$ and $N^{\prime}(t)=\int_{\Omega}u(t)u_t(t)\,dx,$ it follows from the Cauchy-Schwartz inequality and Young's inequality that there exists a positive constant $C,$ such that
\begin{align}\label{2.10}
\nonumber|N^{\prime}(t)|^{\frac{1}{1-\alpha}}&\le \|u_t(t)\|_2^{\frac{1}{1-\alpha}}\|u(t)\|_2^{\frac{1}{1-\alpha}}\\
\nonumber&\le C\|u_t(t)\|_2^{\frac{1}{1-\alpha}}\|u(t)\|_{p_r+1}^{\frac{1}{1-\alpha}}\\
&\le C\left(\|u_t(t)\|_2^2+\|u(t)\|_{p_r+1}^{\frac{2}{1-2\alpha}}\right).
\end{align}
Denote by
\begin{align*}
\beta:=p_r+1-\frac{2}{1-2\alpha},
\end{align*}
and then $\beta>0$. Since $H(t)$ is non-decreasing for any $t\in [0,T_{max})$, we conclude from \eqref{2.4} that
\begin{align}\label{eq6.1}
\nonumber\|u(t)\|_{p_r+1}^{\frac{2}{1-2\alpha}}&=\|u(t)\|_{p_r+1}^{-\beta}\|u(t)\|_{p_r+1}^{p_r+1}\\
\nonumber\le& C\max_{1\le i\le r}\left\{H(t)^{-\frac{\beta}{p_i+1}}\right\}\|u(t)\|_{p_r+1}^{p_r+1}\\
\le& C\max_{1\le i\le r}\left\{H(0)^{-\frac{\beta}{p_i+1}}\right\}\|u(t)\|_{p_r+1}^{p_r+1}.
\end{align}  
Along with inequality \eqref{2.10} and inequality \eqref{eq6.1}, it yields
\begin{align*}
|N^{\prime}(t)|^{\frac{1}{1-\alpha}}\le C\left(\|u_t(t)\|_2^2+\|u(t)\|_{p_r+1}^{p_r+1}\right)
\end{align*}
for any $t\in [0,T_{max}).$ Thus, we deduce from  \eqref{Y'} and \eqref{Y} that there exists a positive constant $C_0$ depending on $p_i,$ $q_j,$ $m,$ $k(0)$ and $H(0),$ such that
\begin{align}\label{adsj}
Y(t)^{\frac{1}{1-\alpha}}\le {C_0}Y^{\prime}(t)
\end{align}
for any $t\in [0,T_{max}),$ which implies that $T_{max}$ is finite and
\begin{align*}
T_{max}<\frac{1-\alpha}{\alpha}C_0Y(0)^{\frac{\alpha}{\alpha-1}}\le \frac{1-\alpha}{\alpha}2C_0H(0)^{-\alpha}.
\end{align*}
\end{proof}
\begin{remark}\label{rmk:6.2}
If we replace the condition $E(0)<0$ in Theorem \ref{negative} with $E(t)<0$ for some $t\in [0,T_{\text{max}})$, the conclusion of Theorem \ref{negative} still holds.
\end{remark}
\subsection{Blow-up of solutions with lower nonnegative initial energy}
The main objective of this subsection is to establish a finite-time blow-up result for problem \eqref{equation} in the case that the initial total energy $E(0)$ is nonnegative.
	
Denote by
\begin{align*}
\ell_0:=\max\{q_l,\sqrt{k(0)}\},
\end{align*}
then it is easy to prove that problem
\begin{align}\label{y*}
\sum_{i=1}^r\frac{a_i(p_1+1)}{p_i+1}\left(2\gamma_i^2\right)^{\frac{p_i+1}{2}}y^{\frac{p_i-1}{2}}=\ell_0+1
\end{align}
admits a unique solution $y_*$ on $\R^+$ and it follows from the assumption $k(0)>1$ that
\begin{align*}
\sum_{i=1}^{r}a_i\left(2\gamma_i^2\right)^{\frac{p_i+1}{2}}{y_*}^{\frac{p_i-1}{2}}\ge \ell_0+1>2.
\end{align*}
Since the function $y\mapsto\sum\limits_{i=1}^{r}a_i\left(2\gamma_i^2\right)^{\frac{p_i+1}{2}}{y}^{\frac{p_i-1}{2}}$ is monotone increasing on $\R^+,$ we deduce from \eqref{y0} that  
\begin{align}\label{y>}
y_*>y_0>0.
\end{align} 
Define
\begin{align*}
M=G(y_*)=y_*-\sum_{i=1}^{r}\frac{a_i}{p_i+1}\left(2\gamma_i^2y_*\right)^{\frac{p_i+1}{2}},
\end{align*}
then we have
\begin{align*}
M=\frac{(p_1-\ell_0)y_*}{p_1+1}.
\end{align*}
\begin{theorem}\label{theorem2}
Assume that Assumption \ref{Assumption1} holds, $p_1>\ell_0,$ $p_r>m,$ $\E(0)>y_0$ and $0\le E(0)<M.$ Then the weak solution $u$ of problem \eqref{equation} blows up in finite time. More precisely, there exists some time $T_{max}\in (0,+\infty),$ such that 
\begin{align*}
\varlimsup_{t\rightarrow T^-_{max}}\|\nabla u(t)\|_2=+\infty.
\end{align*}  
\end{theorem}
\begin{proof}
Thanks to
\begin{align*}
E(t)=\E(t)-\int_{\Omega}F(u(t))u(t)\,dx
\end{align*}
and
\begin{align*}
G(y)=y-\sum_{i=1}^{r}\frac{a_i}{p_i+1}(2\gamma_iy)^{\frac{p_i+1}{2}},
\end{align*}
we obtain from \eqref{quadraticenergyfunctional} that
\begin{align}\label{15}
\nonumber E(t)=&\mathscr{E}(t)-\sum_{i=1}^{r}\frac{a_i}{p_i+1}\|u(t)\|_{p_i+1}^{p_i+1}+\sum_{j=1}^{l}\frac{b_j}{q_j+1}\|u(t)\|_{q_j+1}^{q_j+1}\\
\nonumber\ge& \E(t)-\sum_{i=1}^{r}\frac{a_i\gamma_i^{p_i+1}}{p_i+1}\|\nabla u(t)\|_2^{p_i+1}\\
\ge& \E(t)-\sum_{i=1}^{r}\frac{a_i}{p_i+1}\left(2\gamma_i^2\E(t)\right)^{\frac{p_i+1}{2}}=G(\E(t))
\end{align}
for any $t\in[0,T_{max})$, where the constants $\gamma_i$ are defined as in \eqref{gamma}. Since the function $G(y)$ attains its maximum value at $y=y_0$ and $G^{\prime}(y)<0$ for any $y>y_0,$ then there exists a unique $y_1>y_*,$ such that $G(y_1)=E(0)<M=G(y_*).$ Therefore, we deduce from \eqref{E'} and \eqref{15} that 
\begin{align}\label{y1e}
M>G(y_1)=E(0)\ge E(t)\ge G(\E(t))
\end{align}
for any $t\in [0,T_{max}).$ In particular, we obtain
\begin{align*}
G(y_1)\ge G(\E(0)).
\end{align*}
Thus, we conclude from the assumption $\E(0)>y_0$ that 
\begin{align*}
\E(0)\ge y_1.
\end{align*}
Since $\E(t)$ is continuous in $t$ and $\E(0)\ge y_1,$ it follows from inequality \eqref{y1e} that
\begin{align}\label{>M}
\E(t)\ge y_1>y_*>G(y_*)=M
\end{align}
for any $t\in [0,T_{max}).$ Consequently, we obtain
\begin{align}\label{pqdayu}
\int_{\Omega}F(u(t))\,dx=\E(t)-E(t)\ge y_1-G(y_1)=\sum_{i=1}^r\frac{a_i}{p_i+1}\left(2\gamma_i^2y_1\right)^{\frac{p_i+1}{2}}.
\end{align}
For any $t\in [0,T_{max}),$ denote by
\begin{align*}
\G(t)=M-E(t)
\end{align*}
and
\begin{align*}
N(t)=\frac{1}{2}\|u(t)\|^2_2.
\end{align*}		
 Following the same strategy as in Theorem \ref{negative}, we only need to show that
\begin{align*}
\Y(t)=\G(t)^{1-\alpha}+\epsilon N^{\prime}(t)
\end{align*}
blows up in finite time. We first deduce from inequality \eqref{>M} and Sobolev inequality that
\begin{align*}
\G(t)=&M-E(t)=M-\E(t)+\sum_{i=1}^{r}\frac{a_i}{p_i+1}\|u(t)\|_{p_i+1}^{p_i+1}-\sum_{j=1}^{l}\frac{b_j}{q_j+1}\|u(t)\|_{q_j+1}^{q_j+1}\\
\le&\sum_{i=1}^{r}\frac{a_i}{p_i+1}\|u(t)\|_{p_i+1}^{p_i+1}\le C\sum_{i=1}^r\|u(t)\|_{p_r+1}^{p_i+1}\le C\max_{1\le i\le r}\left\{\|u(t)\|^{p_i+1}_{p_r+1}\right\},
\end{align*}
which implies that there exists a positive constant $c,$ such that
\begin{align}\label{2.5}
\max_{1\le i\le r}\left\{\G(t)^{-\frac{1}{p_i+1}}\right\}\ge c\|u(t)\|^{-1}_{p_r+1}.
\end{align}
Let
\begin{align*}
0<\alpha<\min\left\{\frac{p_r-m}{(p_r+1)(m+1)},\frac{p_r-1}{2(p_r+1)}\right\}<\frac{1}{2},
\end{align*}
by carrying out the similar proof of inequality \eqref{2.13}, we obtain
\begin{align}\label{2.14}
\nonumber N^{\prime\prime}(t)\ge &(1+\delta)\|u_t(t)\|_2^2-2\delta E(t)+\left(\delta-\frac{k(0)-1}{4\delta}-1\right)\|\nabla u(t)\|_2^2\\
\nonumber&+\sum_{i=1}^{r-1}\frac{a_i(p_i+1-2\delta)}{p_i+1}\|u(t)\|_{p_i+1}^{p_i+1}-\sum_{j=1}^{l}\frac{b_j(q_j+1-2\delta)}{q_j+1}\|u(t)\|_{q_j+1}^{q_j+1}\\
\nonumber&+\left(a_r-\lambda\max_{1\le i\le r}\left\{\G(0)^{\frac{m-p_r}{(p_i+1)(m+1)}}\right\}-\frac{2a_r\delta}{p_r+1}\right)\|u(t)\|_{p_r+1}^{p_r+1}\\
&-C_{\lambda}\max_{1\le i\le r}\left\{\G(0)^{\frac{m-p_r}{(p_i+1)(m+1)}+\alpha}\right\}\G(t)^{-\alpha}\G^{\prime}(t)
\end{align}
for any $t\in [0,T_{max}).$ Let
\begin{align*}
\delta=\frac{\ell_0+1}{2}
\end{align*}
and let $\lambda$ be such that
\begin{align*}
\lambda\max_{1\le i\le r}\left\{\G(0)^{\frac{m-p_r}{(p_i+1)(m+1)}}\right\}=\frac{a_r(p_r+1-2\delta)}{2(p_r+1)},
\end{align*}
then we have
\begin{align*}
\delta-\frac{k(0)-1}{4\delta}-1\ge0,\,\,\, q_j+1-2\delta\le0,\,\,j=1,2,\cdots,l,
\end{align*}
which implies that
\begin{align}\label{525}
\nonumber N^{\prime\prime}(t)\ge& \left(\frac{l_0+3}{2}\right)\|u_t(t)\|_2^2-\left(l_0+1\right) E(t)+\sum_{i=1}^{r}\frac{a_i(p_i-l_0)}{p_i+1}\|u(t)\|_{p_i+1}^{p_i+1}\\
\nonumber&-\lambda\max_{1\le i\le r}\left\{\G(0)^{\frac{m-p_r}{(p_i+1)(m+1)}}\right\}\|u(t)\|_{p_r+1}^{p_r+1}\\&-C_{\lambda}\max_{1\le i\le r}\left\{\G(0)^{\frac{m-p_r}{(p_i+1)(m+1)}+\alpha}\right\}\G(t)^{-\alpha}\G^{\prime}(t).
\end{align}
It follows from inequality \eqref{pqdayu} and the definition of $F(u)$ that
\begin{align}\label{526}
\sum_{i=1}^{r}\frac{a_i(p_i-\ell_0)}{p_i+1}\|u(t)\|_{p_i+1}^{p_i+1}\ge(p_1-\ell_0) \int_{\Omega}F(u(t))\,dx\ge \sum_{i=1}^{r}\frac{a_i(p_1-\ell_0)}{p_i+1}\left(2\gamma_i^2y_1\right)^{\frac{p_i+1}{2}}.
\end{align}
We infer from $y_1>y_*,$ $p_1>\ell_0$ and equality \eqref{y*} that
\begin{align}\label{527}
\nonumber M=G(y_*)=&y_*-\sum_{i=1}^{r}\frac{a_i}{p_i+1}\left(2\gamma_i^2y_*\right)^{\frac{p_i+1}{2}}\\
\nonumber=&\frac{p_1-\ell_0}{\ell_0+1}\sum_{i=1}^{r}\frac{a_i}{p_i+1}\left(2\gamma_i^2y_*\right)^{\frac{p_i+1}{2}}\\
<&\frac{1}{\ell_0+1}\sum_{i=1}^{r}\frac{a_i(p_1-\ell_0)}{p_i+1}\left(2\gamma_i^2y_1\right)^{\frac{p_i+1}{2}}.
\end{align} 
Thus, it follows from inequalities \eqref{526} and \eqref{527} that there exists a constant $\theta\in (0,1),$ such that
\begin{align}\label{529}
\left(\ell_0+1\right)E(t)\le (\ell_0+1)E(0)<(\ell_0+1)M=\theta \sum_{i=1}^{r}\frac{a_i(p_i-\ell_0)}{p_i+1}\|u(t)\|_{p_i+1}^{p_i+1}.
\end{align}
Let $\lambda>0$ be a constant such that
\begin{align}\label{lll}
\lambda\max_{1\le i\le r}\left\{\G(0)^{\frac{m-p_r}{(p_i+1)(m+1)}}\right\}=\frac{(1-\theta)a_r(p_r-\ell_0)}{2(p_r+1)}
\end{align}
and let $\epsilon>0$ be a sufficiently small constant, such that 
\begin{align}\label{aaa}
\epsilon \le \min\left\{\frac{1-\alpha}{\max_{1\le i\le r}\left\{\G(0)^{\frac{m-p_r}{(p_i+1)(m+1)}+\alpha}\right\}},1\right\},
\end{align} 
then we conclude from \eqref{526} and \eqref{529}-\eqref{aaa} that there exists a positive constant $c,$ such that
\begin{align}\label{last}
\Y^{\prime}(t)=(1-\alpha)\G(t)^{-\alpha}\G^{\prime}(t)+\epsilon N^{\prime\prime}(t)\ge c\left(\|u_t(t)\|_2^2+\|u(t)\|_{p_r+1}^{p_r+1}\right).
\end{align}
Employing \eqref{2.5} and the monotone increasing property of $\G(t),$ it yields
\begin{align}\label{kkk}
\|u(t)\|_{p_r+1}^{p_r+1}\ge c\min_{1\le i\le r}\left\{\G(t)^{\frac{p_r+1}{p_i+1}}\right\}\ge c\min_{1\le i\le r}\left\{\G(0)^{\frac{p_r-p_i}{p_i+1}}\right\}\G(t).
\end{align}
Along with inequalities	\eqref{last}-\eqref{kkk}, we conclude that there exists a positive constant $c,$ such that
\begin{align*}
\Y^{\prime}(t)\ge c\left(\|u_t(t)\|_2^2+\|u(t)\|_{p_r+1}^{p_r+1}+\G(t)\right).
\end{align*}
If $N^{\prime}(0)<0,$ we shall impose an extra restriction on $\epsilon:$
\begin{align*}
0<\epsilon\le -\frac{\G(0)^{1-\alpha}}{2N^{\prime}(0)}.
\end{align*}
By carrying out the similar proof of Theorem \ref{negative}, we can complete the proof of Theorem \ref{theorem2}.
\end{proof}
\begin{corollary}\label{cor:5.3}
Assume that Assumption \ref{Assumption1} holds, $0\le E(0)<M,$ $p_1>\ell_0$ and $p_r>m.$ If $u_0\in \W_2$, where $\W_2$ is defined as in \eqref{W_2}, then the weak solution of problem \eqref{equation} will blows up in finite time.
\end{corollary}
\begin{proof}
Since $u_0\in \W_2$, we have
\begin{align*}
\|\nabla u_0(0)\|^2_2<\sum_{i=1}^{r}a_i\|u_0(0)\|_{p_i+1}^{p_i+1}\le \sum_{i=1}^ra_i\gamma_i^{p_i+1}\|\nabla u_0(0)\|_2^{p_i+1},
\end{align*}
which implies that
\begin{align*}
2y_0<\|\nabla u_0(0)\|_2^2\le 2\E(0).
\end{align*}
Thus, we conclude from Theorem \ref{theorem2} that the weak solution of problem \eqref{equation} will blows up in finite time. 
\end{proof}
\begin{remark}
We conclude from Remark \ref{rmk:44} that Corollary \ref{cor:5.3} is still true, if we replace $0 \le E(0) < M$ by the assumptions that $0 \le E(0) < d$ and $E(t) < M$ for some time $t \in [0, T_{\max}).$ 
\end{remark}
\begin{remark}
Recall the definition of another potential well $\mathcal{V}$ given in Remark \ref{rmk:3.9}, then Corollary \ref{cor:5.3} holds if the condition $u_0\in \W_2$ is replaced by the assumption 
\begin{align*}
u_0(0)\in \mathcal{V}_2:=\left\{u\in \mathcal{V}:\|\nabla u\|_2^2+\sum_{j=1}^{l_0}b_j\|u\|_{q_j+1}^{q_j+1}<\sum_{i=1}^{r}a_i\|u\|_{p_i+1}^{p_i+1}\right\}.
\end{align*}
\end{remark}
\subsection{Blow-up of solutions with arbitrary initial energy}
In this subsection, we will prove there is no solution of problem \eqref{equation} with arbitrary initial energy defined on $[0,+\infty).$
\begin{theorem}\label{thm:6.6}
Assume that Assumption \ref{Assumption1} holds, $p_1>\ell_0$ and $p_r>m.$ If there exists a positive constant $\eta$ defined as in \eqref{eq:rho0}, such that
\begin{align}\label{eq:assbl}
	0\le E(0)<\eta\int_{\Omega}u_0(0)\partial_tu_0(0)\,dx.
\end{align}
Then the weak solution of problem \eqref{equation} will blows up in finite time. More precisely, there exists some time $T_{max}\in (0,+\infty),$ such that
\begin{align*}
	\varlimsup_{t\rightarrow T^-_{max}}\|\nabla u(t)\|_2=+\infty.
\end{align*}   
\end{theorem}
\begin{proof}
We suppose by contradiction that $u$ is a global solution of problem \eqref{equation}, then we conclude from Theorem \ref{negative} that $E(t)\geq 0$ for any $t\geq 0.$ Denote by 
\begin{align*}
H(t)=-E(t),\,\,\,N(t)=\frac{1}{2}\|u(t)\|_2^2
\end{align*}
and define
\begin{align*}
Z(t):=H(t)+\eta N'(t),
\end{align*}
then it follows from \eqref{twoorderdif}, \eqref{w} and \eqref{ww} that
\begin{align}\label{eq:6.36}
	N^{\prime\prime}(t)\ge& 2\delta H(t)+(1+\delta)\|u_t(t)\|_2^2+\left(\delta-\frac{k(0)-1}{4\delta}-1\right)\|\nabla u(t)\|_2^2-\int_{\Omega}|u_t(t)|^{m-1}u_t(t)u(t)\,dx\nonumber\\
	&+\sum_{i=1}^{r}\frac{a_i(p_i+1-2\delta)}{p_i+1}\|u(t)\|_{p_i+1}^{p_i+1}-\sum_{j=1}^{l}\frac{b_j(q_j+1-2\delta)}{q_j+1}\|u(t)\|_{q_j+1}^{q_j+1}.
\end{align}
Moreover, we deduce from Young's inequality and H\"{o}lder's inequality that
\begin{align*}
	&\left|\int_{\Omega}u|u_t|^{m-1}u_t\,dx\right|\\\le& \|u\|_{m+1}\|u_t\|^m_{m+1}\\
	\le& \frac{m}{m+1}\rho^{\frac{m+1}{m}}\|u_t\|^{m+1}_{m+1}+\frac{1}{m+1}\left(\frac{1}{\rho}\right)^{m+1}\|u\|_{m+1}^{m+1}\\
	\le &\frac{m}{m+1}\rho^{\frac{m+1}{m}}\|u_t\|^{m+1}_{m+1}+\frac{p_r-m}{(p_r-1)(m+1)}\left(\frac{1}{\rho}\right)^{m+1}\|u\|_2^2+\frac{m-1}{(p_r-1)(m+1)}\left(\frac{1}{\rho}\right)^{m+1}\|u\|_{p_r+1}^{p_r+1}
\end{align*}
for any $\rho>0.$ Thus, we conclude from the above inequality and inequality \eqref{eq:6.36} that
\begin{align}\label{eq:6.37}
	N^{\prime\prime}(t)\ge& 2\delta H(t)+(1+\delta)\|u_t(t)\|_2^2+\left(\delta -\frac{k(0)-1}{4\delta}-1\right)\|\nabla u(t)\|_2^2\nonumber\\
	&-\frac{m}{m+1}\rho^{\frac{m+1}{m}}\|u_t(t)\|_{m+1}^{m+1}-\frac{p_r-m}{(p_r-1)(m+1)}\left(\frac{1}{\rho}\right)^{m+1}\|u(t)\|_2^2\nonumber\\
	&+\sum_{i=1}^{r-1}\frac{a_i(p_i+1-2\delta)}{p_i+1}\|u(t)\|_{p_i+1}^{p_i+1}-\sum_{j=1}^{l}\frac{b_j(q_j+1-2\delta)}{q_j+1}\|u(t)\|_{q_j+1}^{q_j+1}\nonumber\\
	&+\left(\frac{a_r(p_r+1-2\delta)}{p_r+1}-\frac{m-1}{(p_r-1)(m+1)}\left(\frac{1}{\rho}\right)^{m+1}\right)\|u(t)\|_{p_r+1}^{p_r+1}.
\end{align}
Choosing 
\begin{align*}
\delta=\frac{\ell_0+p_1+2}{4},
\end{align*}
then we have
\begin{align}\label{eq:6.40}
	\frac{a_i(p_i+1-2\delta)}{p_i+1}>0,\quad \frac{b_j(q_j+1-2\delta)}{q_j+1}<0\quad\text{and}\quad \delta-\frac{k(0)-1}{4\delta}-1>0
\end{align}
for each $i=1,\cdots, r$ and $j=1,\cdots,l$. Moreover, if we choose 
\begin{align*}
	\rho_0:=\left(\max\left\{\frac{(p_r+1)(m-1)}{a_r(p_r+1-2\delta)(p_r-1)(m+1)},\frac{2\delta\gamma_0^2(p_r-m)}{(4\delta^2-4\delta-k(0)+1)(p_r-1)(m+1)}\right\}\right)^{\frac{1}{m+1}}
\end{align*}
with
\begin{align*}
	\gamma_0=\sup\{\|v\|_{2}:v\in H^1_0(\Omega),\|\nabla v\|_2=1\},
\end{align*}
then we have
\begin{align}\label{eq:6.38}
	&\left(\delta-\frac{k(0)-1}{4\delta}-1\right)\|\nabla u(t)\|_2^2-\frac{p_r-m}{(p_r-1)(m+1)}\left(\frac{1}{\rho_0}\right)^{m+1}\|u(t)\|_2^2\nonumber\\
	\ge&\frac{1}{2\gamma_0^2}\left(\delta-\frac{k(0)-1}{4\delta}-1\right)\|u(t)\|_2^2
\end{align}
and 
\begin{align}\label{eq:6.39}
	\frac{a_r(p_r+1-2\delta)}{p_r+1}-\frac{m-1}{(p_r-1)(m+1)}\left(\frac{1}{\rho_0}\right)^{m+1}\ge0.
\end{align}
Let
\begin{align}\label{eq:rho0}
0<\eta\leq\frac{m+1}{m}\rho_0^{-\frac{m+1}{m}}
\end{align}
be a sufficiently small constant, such that
\begin{align*}
\eta\delta\le \min\left\{1+\delta,\frac{1}{2\gamma_0^2}\left(\delta-\frac{k(0)-1}{4\delta}-1\right)\right\}
\end{align*}
and let $\zeta>0$ be a constant with the following properties
\begin{align*}
	\zeta\geq \eta\delta,\quad \zeta\le 1+\delta\quad \text{and}\quad \zeta\le \frac{1}{2\gamma_0^2}\left(\delta-\frac{k(0)-1}{4\delta}-1\right),
\end{align*}
we conclude from \eqref{E'}, \eqref{eq:6.37}-\eqref{eq:6.39}, $\mu'(s)<0$ and $H(t)=-E(t)\le0$ for any $t\geq 0$ that
\begin{align}\label{eq:Z'}
\nonumber\frac{d}{dt}Z(t)\ge& 2\eta\delta H(t)+\eta(1+\delta)\|u_t(t)\|_2^2+\frac{\eta}{2\gamma_0^2}\left(\delta-\frac{k(0)-1}{4\delta}-1\right)\|u(t)\|_2^2\\
\nonumber\ge& \zeta (2H(t)+\eta\|u_t(t)\|_2^2+\eta\|u(t)\|_2^2)\\
\ge& 2\zeta Z(t).
\end{align}
Hence, we infer from the classical Gronwall inequality and inequality \eqref{eq:Z'} that
\begin{align}\label{eq:6.43}
	Z(t)\ge e^{\zeta t}Z(0)
\end{align}
for any $t\geq 0.$

On the other hand, we deduce from H\"{o}lder's inequality and equality \eqref{Eenergyidentity} that
\begin{align}\label{eq:6.42}
	\int_{0}^{t}Z(\tau)\,d\tau&=\int_{0}^{t}H(\tau)\,d\tau+\frac{\eta}{2} \|u(t)\|_2^2-\frac{\eta}{2}\|u_0(0)\|_2^2\nonumber\\
	&
	\le \int_{0}^{t}H(\tau)\,d\tau+\eta\|u_0(0)\|_2^2+\eta\left(\int_{0}^{t}\|u_t(\tau)\|_2\,d\tau\right)^2-\frac{\eta}{2}\|u_0(0)\|_2^2\nonumber\\
	&\le \int_{0}^{t}H(\tau)\,d\tau+\frac{\eta}{2}\|u_0(0)\|_2^2+\eta|\Omega|^{\frac{m-1}{m+1}}t^{\frac{2m}{m+1}}\left(\int_{0}^{t}\|u_t(\tau)\|_{m+1}^{m+1}\,d\tau\right)^{\frac{2}{m+1}}\nonumber\\&
	\le \int_{0}^{t}H(\tau)\,d\tau+\frac{\eta}{2}\|u_0(0)\|_2^2+\eta|\Omega|^{\frac{m-1}{m+1}}t^{\frac{2m}{m+1}}\left(E(0)-E(t)\right)^{\frac{2}{m+1}}.
\end{align}
It follows from $H(t)=-E(t)\le 0$ for any $t\geq 0$ and inequalities \eqref{eq:6.43}-\eqref{eq:6.42} that
\begin{align}\label{eq:6.44}
	\frac{1}{\zeta}\left(e^{\zeta t}-1\right)Z(0)\le\int_{0}^{t}Z(\tau)\,d\tau\le \frac{\eta}{2}\|u_0(0)\|_2^2+\eta|\Omega|^{\frac{m-1}{m+1}}t^{\frac{2m}{m+1}}\left(E(0)\right)^{\frac{2}{m+1}}
\end{align}
for all $t\geq 0.$ Since $Z(0)>0,$ we obtain a contradiction. Thus, the solution of problem \eqref{equation} will blow up in a finite time. 
\end{proof}
In what follows, we will prove the existence of the finite time blow-up solutions of problem \eqref{equation} with arbitrarily high initial energy. 
\begin{theorem}
Assume that Assumption \ref{Assumption1} holds, $p_1>\ell_0$ and $p_r>m.$ For any $R\in \R,$ there exists initial data $u_0$ satisfying Assumption \ref{Assumption1}, such that $E(0)=R$ and the weak solution $u(t)$ of problem \eqref{equation} with initial data $u_0$ blows up in a finite time.
\end{theorem}
\begin{proof}
\begin{enumerate}[\rm{(}1\rm{)}]
\item If $R<0,$ then for any fixed $v \in L^2_{\mu}(\mathbb{R}^-;H^1_0(\Omega))$ with $\partial_t v \in L^2_{\mu}(\mathbb{R}^-;L^2(\Omega)),$ such that $v,$ $\partial_t v$ are weakly continuous at $t=0$ and $\|\nabla v(0)\|_2 > 0,$ there exists a $\lambda_1=\lambda_1(R,v)>0,$ such that
\begin{align*}
	&\frac{1}{2}\left(\|\lambda\partial_tv(0)\|_2^2+\|\lambda\nabla v(0)\|_2^2+\int_{0}^{+\infty}\mu(s)\|\lambda\nabla \left(v(0)-v(-s)\right)\|_2^2\,ds\right)\\
	&-\sum_{i=1}^{r}\frac{a_i}{p_i+1}\|\lambda v(0)\|_{p_i+1}^{p_i+1}+\sum_{j=1}^{l}\frac{b_j}{q_j+1}\|\lambda v(0)\|_{q_j+1}^{q_j+1}=R
\end{align*}
Let $u_0=\lambda_1v,$ then we have $E(0)=R<0.$ Thus, we conclude from Theorem \ref{negative} that the weak solution of problem \eqref{equation} with initial data $u_0$ blows up in a finite time.

\item If $R\ge0.$ Thanks to Theorem \ref{thm:6.6}, it suffice to prove that there exists a function $u_0$ satisfy Assumption \ref{Assumption1} and
\begin{align*}
	R=E(0)<\eta\int_{\Omega}u_0(0)\partial_t u_0(0)\,dx,
\end{align*}
where $\eta$ is the same as in Theorem \ref{thm:6.6}.

Let $\Omega_1,$ $\Omega_2\subset\Omega$ be two disjoint domains. Denote by $\xi\in C^{\infty}(\R^-)$ a smooth function with the following properties
\begin{align*}
\xi(0)>0,\,\,\,\xi'(0)>0,\,\,\,\xi(t)=0,\quad \text{for}\ t\le -1
\end{align*}
and let $\varphi\in C_c^{\infty}(\Omega)$ be a nonzero function with the support of $\varphi$ satisfying $\spt\varphi\subset\Omega_1.$ Then for any $\lambda>0$, the functions $\varphi_{\lambda}(x,t)=\lambda \xi(t)\varphi(x)\in L_{\mu}^2(\mathbb{R}^-;H^1_0(\Omega))$ and $\partial_t\varphi_{\lambda}(x,t)=\lambda \xi'(t)\varphi(x)\in L_{\mu}^2(\mathbb{R}^-;L^2(\Omega))$ are weakly continuous at $t=0.$ Moreover, it is easy to prove that there exists a constant $\lambda_1>0,$ such that for any $\lambda>\lambda_1$, 
\begin{align}\label{eq:6.47}
	\int_{\Omega}\varphi_{\lambda}(x,0)\partial_t\varphi_\lambda(x,0)\,dx=\lambda^2\xi(0)\xi'(0)\|\varphi\|_2^2>\frac{R}{\eta}.
\end{align}
Define a functional $\mathcal{L}$ on $H^1_0(\Omega)$ by
\begin{align*}
	\mathcal{L}(\phi):=&\frac{1}{2}|\xi'(0)|^2\|\phi\|_2^2+\frac{1}{2}\left(|\xi(0)|^2+\int_{0}^{+\infty}\mu(s)|\xi(0)-\xi(-s)|^2\,ds\right)\|\nabla \phi\|_2^2\\
	&-\sum_{i=1}^{r}\frac{a_i|\xi(0)|^{p_i+1}}{p_i+1}\|\phi\|_{p_i+1}^{p_i+1}+\sum_{j=1}^{l}\frac{b_j|\xi(0)|^{q_j+1}}{q_j+1}\|\phi\|_{q_j+1}^{q_j+1},\,\,\,\forall\,\,\,\phi\in H_0^1(\Omega),
\end{align*}
then there exists $\lambda_2\ge\lambda_1,$ such that
\begin{align*}
	\mathcal{L}(\lambda\varphi)<R-1
\end{align*}
for any $\lambda>\lambda_2.$

On the other hand, let $\{e_k\}_{k=1}^{\infty}\subset H_0^1(\Omega_2)$ be an orthonormal basis, we deduce from Sobolev embedding theorem that
\begin{align*}
e_k\rightharpoonup 0\quad\text{in}\ H^1_0(\Omega_2)\end{align*}
and
\begin{align*}
e_k\to 0\quad \text{in}\ L^{p_i+1}(\Omega_2)
\end{align*}
for any $i=1,$ $\cdots,$ $r.$ Thus, there exists a constant $R_0>0$ and an integer $N\in \mathbb{N},$ such that
\begin{align*}
	&\frac{1}{2}\left(|\xi(0)|^2+\int_{0}^{+\infty}\mu(s)|\xi(0)-\xi(-s)|^2\,ds\right)R_0^2>R+1-\mathcal{L}(\lambda_2 \varphi),\\
	&\sum_{i=1}^{r}\frac{a_i|\xi(0)|^{p_i+1}R_0^{p_i+1}}{p_i+1}\|e_N\|_{L^{p_i+1}(\Omega_2)}^{p_i+1}<\frac{1}{2},
\end{align*}
which implies that
\begin{align*}
\mathcal{L}(R_0\tilde{e}_N+\lambda_2\varphi)=\mathcal{L}(R_0\tilde{e}_N)+\mathcal{L}(\lambda_2\varphi)>R+\frac{1}{2},
\end{align*}
where 
\begin{equation*}
\tilde{e}_N=
\begin{cases}
e_N,\,\,\,x\in \Omega_2,\\
0,\,\,\,\,x\in \Omega\setminus\Omega_2.
\end{cases}
\end{equation*}
Since $\mathcal{L}(\lambda\varphi)$ is continuous in $\lambda\in [0,+\infty)$ and $\mathcal{L}(\lambda\varphi)\to -\infty$ as $\lambda\to +\infty$, there exists $\bar\lambda>\lambda_2,$ such that
\begin{align}\label{eq:6.48}
\mathcal{L}(R_0\tilde{e}_N+\bar\lambda\varphi)=	\mathcal{L}(R_0\tilde{e}_N)+\mathcal{L}(\bar\lambda\varphi)=R.
\end{align}
Denote by
\begin{align*}
u_0(x,t)=\bar\lambda\xi(t)\varphi(x)+R_0\xi(t)\tilde{e}_N(x),
\end{align*}
then $u_0$ satisfies the conditions of Assumption \ref{Assumption1}. Moreover, it follows from \eqref{eq:6.47}-\eqref{eq:6.48} and $\spt\varphi\cap\spt\tilde{e}_N=\emptyset$ that
\begin{align*}
	E(0)=\mathcal{L}(R_0\tilde{e}_N)+\mathcal{L}(\bar\lambda\varphi)=R\le \eta\bar{\lambda}^2\xi(0)\xi'(0)\|\varphi\|_2^2<\eta\int_{\Omega}u_0(0)\partial_t u_0(0)\,dx.
\end{align*}
\end{enumerate}
\end{proof}
\section*{Acknowledgement}
Research of Bo You partially supported by the National Science Foundation of China Grant (11871389), the Fundamental Research Funds for the Central Universities (xzy012022008) and Shaanxi Fundamental Science Research Project for Mathematics and Physics (22JSY032). Research of Marcelo M. Cavalcanti partially supported by the CNPq Grant 300631/2003-0.

	\bibliography{BIB.bib}
	\bibliographystyle{abbrv}
\end{document}